\documentclass[a4paper,11pt]{amsart}

\usepackage{latexsym}
\usepackage{a4wide}
\usepackage{amscd}
\usepackage{graphics}
\usepackage{amsmath}
\usepackage{amssymb}
\usepackage{mathrsfs}

\newcommand{\R}{\mathbb R}
\newcommand{\C}{\mathbb C}
\newcommand{\N}{\mathbb N}
\newcommand{\Z}{\mathbb Z}
\newcommand{\A}{\mathcal A}
\newcommand{\sphere}[1]{\mathbb{S}^{#1}}
\newcommand{\e}{\mathrm e}
\newcommand{\eps}{\varepsilon}
\newcommand{\abs}[1]{\left\vert #1 \right\vert}
\newcommand{\vett}[1]{\overrightarrow{#1}}
\newcommand{\nor}[1]{\left\Vert #1 \right\Vert}
\newcommand{\cinf}[1]{\textit{C}^{\infty}_{C}(#1)}
\newcommand{\dom}{\textit{D}^{1,2}(\R^{N})}
\newcommand{\domk}{\textit{D}_k^{1,2}(\R^{N})}
\newcommand{\domdr}{\textit{D}_{r_1,r_2}^{1,2}(\R^{N})}
\newcommand{\qa}{\int_{\R^{N}}\bigg\vert\bigg(i\nabla-\frac{A(\theta)}{\abs{x}}\bigg)u\bigg\vert^2}

\newcommand{\curl}[1]{\mathrm{curl}#1}

\newtheorem{thm}{\bf Theorem}[section]      
\newtheorem{lem}[thm]{\bf Lemma}            
\newtheorem{prop}[thm]{\bf  Proposition}     
\newtheorem{defn}[thm]{\bf Definition}      
\newtheorem{properties}[thm]{\bf Properties}      
\newtheorem{rem}[thm]{\bf Remark}       

\begin{document}

\title[Solutions to NLS with singular potentials and critical exponent]{Solutions to nonlinear Schr\"{o}dinger equations with singular electromagnetic potential and critical exponent}
\author{Laura Abatangelo, Susanna Terracini}

\address{L. Abatangelo and S. Terracini: Dipartimento di Matematica e Applicazioni, Universit\`a di Milano Bicocca, Piazza Ateneo Nuovo, 1, 20126 Milano (Italy)}
\email{l.abatangelo@campus.unimib.it, susanna.terracini@unimib.it}


\date{\today}
\maketitle

\begin{abstract}
\noindent We investigate existence and qualitative behaviour of
solutions to nonlinear Schr\"{o}dinger equations with critical exponent and singular
electromagnetic potentials. We are concerned with with magnetic vector potentials which are
homogeneous of degree $-1$, including the Aharonov-Bohm class. In particular,
 by variational arguments we prove
a result of multiplicity of solutions distinguished by symmetry
properties.
\end{abstract}

\section{Introduction}\label{sec:intro}
In norelativistic quantum mechanics, the Hamiltonian associated with a charged particle in an 
electromagnetic field is given by $(i\nabla - A)^2+V$ 
where $A:\ \R^N\rightarrow\R^N$ is the magnetic potential and 
$V:\ \R^N\rightarrow\R$ is the electric one.  The vector field $B=\curl{A}$ has to be 
intended as  the differential $2$--form $B=da$,  $a$ being the $1$--form canonically 
associated with the vector field $A$. Only in three dimensions, by duality,
$B$ is represented by another vector field.

In this paper we are concerned with differential operators of the form
$$
\left(i\nabla-\frac{A(\theta)}{\abs{x}}\right)^2-\frac{a(\theta)}{\abs{x}^2}
$$
where $A(\theta)\in
L^{\infty}(\sphere{N-1},\R^{N})$ and $a(\theta)\in L^{\infty}(\sphere{N-1},\R)$. Notice the 
presence of homogeneous (fuchsian) singularities at the origin. In some situations the 
potentials may also have singularities on the sphere.

This kind of magnetic potentials appear as limits of thin solenoids, when
the circulation remains constant as the sequence  of solenoids' radii tends to zero, 
The limiting vector field is then a singular measure supported in a lower dimensional set. 
Though the resulting magnetic field vanishes almost everywhere, its presence still affects the 
spectrum of the operator,  giving rise to the so-called ``Aharonov-Bohm effect''.

Also from the mathematical point of view this class of operators is 
worty being investigated, mainly because of their critical behaviour. Indeed, 
they share  with the Laplacian the same degree of homogeneity and invariance under the 
Kelvin transform. Therefore they cannot be regarded as lower 
order perturbations of the Laplace operator (they do not belong to the Kato's class:
see fo instance \cite{FFT08}, \cite{FMT08} and references therein).

Here we shall always assume $N\geq3$, otherwise specified. 
A quadratic form is associated with the differential operator, that is
\begin{equation}\label{fq}
\int_{\R^{N}}\bigg\vert\bigg(i\nabla-\frac{A(\theta)}{\abs{x}}\bigg)u\bigg
\vert^2-\int_{\R^{N}}\frac{a(\theta)}{\abs{x}^2}\,u^2.
\end{equation}

As its natural domain we shall take the closure of compactly supported functions
$\cinf{\R^N\setminus\{0\}}$ with respect to the quadratic form itself. 
Thanks to Hardy type inequalities, when $N\geq 3$, this space turns out to be the same 
$\dom$, provided   $a$  
is suitably bounded (\cite{FFT08}), while,  when $N=2$, this is a smaller space of functions vanishing
at the pole of the magnetic potential. Throughout the paper we shall always assume positivity of \eqref{fq}.

We are interested in solutions to the critial semilinear differential equations
\begin{equation}\label{eqintro}
 \left(i\nabla-\frac{A(\theta)}{\abs{x}}\right)^2u-\frac{a(\theta)}{\abs{x}^2}u=\abs{u}^{2^*-2}u \qquad \textrm{in $\R^N\setminus\{0\}$}
\end{equation}
and in particular in their symmetry properties. The critical exponent appears as the 
natural one whenever seeking finite
energy solutions: indeed, Pohozaev type identitities prevent the existence of entire solutions
for power nonlinearities of different degrees.

The first existence results for equations of type (\ref{eqintro}) are given in \cite{EL89} for
subcritical nonlinearities. In addition, existence and multiplicity of solutions are investigated 
for instance in \cite{C03,CIS09,K00,ST02} mainly via variational methods and concentration-compactness 
arguments. Some results involving critical nonlinearities are present in \cite{AS03,ChS05}.
Concerning results on semiclassical solutions we quote \cite{CS02,CS05}.
As far as we know, not many papers are concerned when electromagnetic potentials which 
are singular,   except  those in \cite{K97}, where anyway several integrability hypotheses are assumed on 
them, and, much more related with ours, the paper \cite{CS10} that we discuss later on.
 
We are interested in the existence of solutions to Equation (\ref{eqintro}) 
distinct by symmetry properties,
as it happens in \cite{Ter96} for Schr\"{o}dinger operators when magnetic vector potential
is not present. To investigate these questions, we
aim to extend some of the results contained in \cite{Ter96} when a
singular electromagnetic potential is present.

To do this, we refer to solutions which minimize
the Rayleigh quotient
$$
 \frac{\displaystyle\int_{\R^{N}}\bigg\vert\bigg(i\nabla-\frac{A(\theta)}{\abs{x}}\bigg)u\bigg\vert^2-\int_{\R^{N}}\frac{a(\theta)}{\abs{x}^2}\,u^2}{\bigg(\displaystyle\int_{\R^N}\abs{u}^{2^*}\bigg)^{2/2^*}}\,.$$

We find useful to stress that, although, in general,
 ground states in $\dom$ to equation (\ref{eqintro}) 
do not exist
(see Section 3), the existence of minimizers con be granted in suitable subspaces of symmetric functions.
 
We are concerned with Aharonov-Bohm type potentials too. 
In $\R^2$ a vector potential associated to the Aharonov-Bohm magnetic field has the form
$$\mathcal A(x_1,x_2)=\alpha\left(-\frac{x_2}{\abs{x}^2},\frac{x_1}{\abs{x}^2}\right)$$
where $\alpha\in\R$ stands for the circulation of $\mathcal A$
around the thin solenoid.
Here we consider the analogous of these potentials in $\R^N$ for
$N\geq4$, that is
$$\A(x_1,x_2,x_3)=\left(\frac{-\alpha x_2}{x_1^2+x_2^2}\, , \frac{\alpha x_1}{x_1^2+x_2^2}\, , 0\right) \qquad (x_1,x_2)\in\R^2\, , x_3\in\R^{N-2}\ .$$

Our main result can be stated as follows:
\begin{thm}
 Assume $N\geq4$ and $a(\theta)\equiv a \in \R^-$. There exist $a^*<0$ such that, when $a<a^*$, the equation (\ref{eqintro}) admits at least two distinct solutions in $\dom$: one is radially symmetric while the second one is only invariant under a discrete group of rotations on the first two variables.

A similar result holds for Aharonov-Bohm type potentials.
\end{thm}

We point out hypothesis on the dimension is purely technical here. By the way,
in dimension $N=3$ and in case of Aharonov-Bohm potentials, Clapp
and Szulkin proved in \cite{CS10} the existence of at least a solution
which enjoys the so-called \emph{biradial} symmetry. However, their
argument may be adapted even in further dimensions, provided a
cylindrical symmetry is asked to functions with respect to the
second set of variables in $\R^{N-2}$.

The proof of our main result is based on a comparison between the
different levels of the Rayleigh quotient's infima taken over
different spaces of functions which enjoy certain symmetry
properties. In particular, we will focus our attention on three
different kinds of symmetries:
\begin{enumerate}
 \item functions which are invariant under the $\Z_k\times SO(N-2)$ action for $k\in\N$, $m\in\Z$ defined as
$$u(z,y)\mapsto \e^{i\frac{2\pi}{k}m}u(\e^{i\frac{2\pi}{k}}z,Ry) \qquad \textrm{for $z\in\R^2$ and $y\in\R^{N-2}$, $R\in SO(N-2)$,}$$
$\domk$ will denote their vector space;
 \item functions which we will call "biradial", i.e.
$$u(z,y)= u(r_1,r_2)\qquad \textrm{where $r_1=\sqrt{x_1^2+x_2^2}$ and $r_2=\sqrt{x_3^2+\cdots+x_N^2}$,}$$
$D^{1,2}_{r_1,r_2}$ will denote their vector space; sometimes we shall consider the symmetric functions 
$u(z,y)=\e^{im\arg{z}} u(r_1,r_2)$.
 \item functions which are radial, $D^{1,2}_{rad}$ will be their
 vector space.
\end{enumerate}

We fix the notation we will use throughout the paper:
\begin{defn}
 $S_{A,a}^{r_1,r_2}$ is the minimum of the Rayleigh quotient related to the magnetic Laplacian over all the biradial functions in $\dom$;

$S_{0,a}^{r_1,r_2}$ is the minimum of the Rayleigh quotient related
to the usual Laplacian over all the biradial functions in $\dom$;

$S_{0,a}^{rad}$ is the minimum of the Rayleigh quotient related to
the usual Laplacian over all the radial functions in $\dom$;

$S_{0,a}^{k}$ is the minimum of the Rayleigh quotient related to the
usual Laplacian over all the functions in $\domk$;

$S_{A,a}^{k}$ is the minimum of the Rayleigh quotient related to the
magnetic Laplacian over all the functions in $\domk$;

$S$ is the usual Sobolev constant for the immersion $\dom\hookrightarrow L^{2^*}(\R^N)$.
\end{defn}

In order to prove these quantities are achieved, we use
concentration-compactness arguments, in a special form due to
Solimini in \cite{Sol95}. Unfortunately, we are not able to compute
the precise values of the abovementioned infima, but only to provide
estimates in terms of the Sobolev constant $S$; nevertheless this is
enough to our aims. By the way, it is worth being noticed in
\cite{Ter96} a characterization is given for the radial case
$S_{0,a}^{rad}$: it is proved $S_{0,a}^{rad}$ is achieved and the
author is able to compute its precise value. This will turn out
basic when we compare it with the other infimum values in order to
deduce some results about symmetry properties.

\smallskip

Both in case of $\frac{A(\theta)}{\abs{x}}$ type potentials and
Aharonov-Bohm type potentials, we follow the same outline. We
organize the paper as follows: first of all in Section 2 we state
the variational framework for our problem; secondly in Section 3 we
provide some sufficient conditions to have the infimum of the
Rayleigh quotients achieved, beginning from some simple particular
cases; in Section 4 we investigate the potential symmetry of
solutions; finally in Section 6 we deduce our main result. On the
other hand, Section 5 is devoted to the study of Aharonov-Bohm type
potentials.

\section{Variational setting}

As initial domain for the quadratic form (\ref{fq}) we take the
space of compactly supported functions in $\R^{N}\setminus\{0\}$ :
we denote it $\cinf{\R^{N}\setminus\{0\}}$. Actually, as a
consequence of the following lemmas, one can consider the space
$\dom$ as the maximal domain for the quadratic form. We recall that by definition
$\dom=\overline{\cinf{\R^N}}^{(\int_{\R^N}\abs{\nabla u}^2)^{1/2}}$,
i.e. the completion of the compact supported functions on $\R^N$
under the so-called Dirichlet norm.

The main tools for this are the following basic inequalities:
\begin{eqnarray*}
\int_{\R^N}\frac{u^2}{\abs{x}^2}\,dx&\leq&\frac{4}{(N-2)^2}\int_{\R^N}\abs{\nabla
u}^2\,dx \quad \textrm{(Hardy inequality)}\\
\int_{\R^N}\abs{\nabla\abs{u}}^2\,dx&\leq&\int_{\R^N}\abs{\left(i\nabla
- \frac{A}{\abs{x}}\right)u}^2\,dx \quad \textrm{(diamagnetic
inequality)}
\end{eqnarray*}
both with the following lemmas
\begin{lem}\label{firstlemma}
The completion of $\cinf{\R^{N}\setminus\{0\}}$ under the Dirichlet norm coincide with the space $\dom$.
\end{lem}

\begin{lem}
If $A\in L^{\infty}(\sphere{N-1})$ then the norm $\bigg(\displaystyle\int_{\R^N}\bigg\vert\bigg(i\nabla-\frac{A(\theta)}{\abs{x}}\bigg)u\bigg\vert^2\bigg)^{1/2}$ is equivalent to the Dirichlet norm on $\cinf{\R^N\setminus\{0\}}$.
\end{lem}
%
%

\begin{lem}\label{formaquadfq}
The quadratic form (\ref{fq}) is equivalent to
$Q_A(u)=\displaystyle\int_{\R^N}\bigg\vert\bigg(i\nabla-\frac{A(\theta)}{\abs{x}}\bigg)u\bigg\vert^2$
on its maximal domain $\dom$ provided $\nor{a}_{\infty}<(N-2)^2/4$.
Moreover, it is positive definite.
\end{lem}

We refer to \cite{FFT08} for a deeper analysis of these questions.

\bigskip

\bigskip

We set the following variational problem
\begin{equation}\label{sa}
 S_{A,a}:=\inf_{u\in\dom\setminus\{0\}}\frac{\displaystyle\int_{\R^{N}}\bigg\vert\bigg(i\nabla-\frac{A(\theta)}{\abs{x}}\bigg)u\bigg\vert^2
 -\int_{\R^{N}}\frac{a(\theta)}{\abs{x}^2}\,u^2}{\bigg(\displaystyle\int_{\R^N}\abs{u}^{2^*}\bigg)^{2/2^*}}.
\end{equation}
Of course, $S_{A,a}$ is strictly positive since the quadratic form
(\ref{fq}) is positive definite.

We are now proposing a lemma which will be useful later.
\begin{lem}\label{lemmadivergenzatraslazione}
 Let $\{x_n\}$ be a sequence of points such that $\abs{x_n}\rightarrow\infty$ as $n\rightarrow\infty$.
 Then for any $u\in\dom$ as $n\rightarrow\infty$ we have
 $$
\frac{\displaystyle\int_{\R^{N}}\bigg\vert\bigg(i\nabla-\frac{A(\theta)}{\abs{x}}\bigg)u(\cdot+x_n)\bigg\vert^2
-\int_{\R^{N}}\frac{a(\theta)}{\abs{x}^2}\,\abs{u(\cdot+x_n)}^2}{\bigg(\displaystyle\int_{\R^N}\abs{u}^{2^*}\bigg)^{2/2^*}}
\rightarrow\frac{\displaystyle\int_{\R^N}\abs{\nabla
u}^2}{\left(\displaystyle\int_{\R^N}\abs{u}^{2^*}\right)^{2/{2^*}}}.
 $$
\end{lem}
\begin{proof}
It is sufficient to prove for all $\eps>0$ there exists a
$\overline{n}$ such that
$\int_{\R^N}\frac{\abs{u(x+x_n)}^2}{\abs{x}^2}dx<2\eps$ for
$n\geq\overline{n}$. Let us consider $R>0$ big enough to have
$$\int_{\R^N\setminus
B_R(x_n)}\frac{\abs{u(x+x_n)}^2}{\abs{x}^2}dx<\eps$$ for every $n\in\N$. On the other
hand, when $x\in B_R(x_n)$ we have
$\abs{x}\geq\abs{x_n}-\abs{x-x_n}\geq\abs{x_n}-R$ which is a
positive quantity for $n$ big enough. In this way
$$\int_{B_R(x_n)}\frac{\abs{u(x+x_n)}^2}{\abs{x}^2}dx
\leq\frac{1}{(\abs{x_n}-R)^2}\int_{B_R(x_n)}\abs{u(x+x_n)}^2\,dx<\eps$$ for $n$ big enough. \qed
\end{proof}

\bigskip

Exploiting this lemma, we can state the following property holding
for $S_{A,a}$:
\begin{prop}\label{SA<S}
 If $S$ denotes the best Sobolev constant for the embedding of $\dom$ in $L^{2^*}(\R^N)$, i.e.
\begin{equation}\label{s}
 S:=\inf_{u\in\dom\setminus\{0\}}\frac{\displaystyle\int_{\R^N}\abs{\nabla u}^2}{\bigg(\displaystyle\int_{\R^N}\abs{u}^{2^*}\bigg)^{2/2^*}}\quad,
\end{equation}
it holds $S_{A,a}\leq S$.
\end{prop}
\begin{proof}
Lemma (\ref{lemmadivergenzatraslazione}) shows immediately for all
$u\in\dom$
$$S_{A,a}\leq\frac{\displaystyle\int_{\R^N}\abs{\nabla u}^2}{\left(\displaystyle\int_{\R^N}\abs{u}^{2^*}\right)^{2/{2^*}}}+o(1)\ .$$
If we choose a minimizing sequence for (\ref{s}) in the line above,
we see immediately $S_{A,a}\leq S$. \qed
\end{proof}

\section{Attaining the infimum}

Given the results in \cite{BN83} due to Brezis and Nirenberg, one could expect 
that ,if $S_{A,a}$ is strictly less than $S$, then it is attained. Here we pursue 
this idea with concentration-compactness arguments, in the special version 
due to Solimini in \cite{Sol95}.
Before proceeding, we find useful to
recall some definitions about the so-called \emph{Lorentz spaces}.

\begin{defn}\cite{Sol95}
A \emph{Lorentz space} $L^{p,q}(\R^N)$ is a space of measurable
functions affected by two indexes $p$ and $q$ which are two positive
real numbers, $1\leq p\,,q\leq+\infty$, like the indexes which
determine the usual $L^p$ spaces. The index $p$ is called
\emph{principal index} and the index $q$ is called \emph{secundary
index}. A monotonicity property holds with respect to the secundary
index: if $q_1<q_2$ then $L^{p,q_1}\subset L^{p,q_2}$. So the
strongest case of a Lorentz space with principal index $p$ is
$L^{p,1}$; while the weakest case is $L^{p,\infty}$, which is
equivalent to the so-called \emph{weak $L^p$ space}, or
Marcinkiewicz space. Anyway, the most familiar case of Lorentz space
is the intermediate case given by $q=p$, since the space $L^{p,p}$
is equivalent to the classical $L^p$ space.
\end{defn}

\begin{properties}\cite{Sol95}
A basic property about the Lorentz spaces is an appropriate case of
the H\"{o}lder inequality, which states that the duality product of
two functions is bounded by a constant times the product of the
norms of the two functions in two respective conjugate Lorentz
spaces $L^{p_1,q_1}$ and $L^{p_2,q_2}$ where the two pairs of
indexes satisfy the relations
$\frac{1}{p_1}+\frac{1}{p_2}=\frac{1}{q_1}+\frac{1}{q_2}=1$.

Moreover, if we consider the Sobolev space $H^{1,p}(\R^N)$, it is
wellknown it is embedded in the Lebesgue space $L^{p^*}(\R^N)$. But
this embedding is not optimal: it holds that the space
$H^{1,p}(\R^N)$ is embedded in the Lorentz space $L^{p^*,p}$, which
is strictly stronger than $L^{p^*}=L^{p^*,p^*}$.
\end{properties}

\begin{thm}\textbf{(Solimini)}\label{solimini}(\cite{Sol95})
Let $(u_n)_{n\in\N}$ be a given bounded sequence of functions in $H^{1,p}(\R^N)$, with the index $p$ satisfing $1<p<N$. Then, replacing $(u_n)_{n\in\N}$ with a suitable subsequence, we can find a sequence of functions $(\phi_i)_{i\in\N}$ belonging to $H^{1,p}(\R^N)$ and, in correspondence of any index $n$, we can find a sequence of rescalings $(\rho_n^i)_{i\in\N}$ in such a way that the sequence $(\rho_n^i(\phi_i))_{i\in\N}$ is summable in $H^{1,p}(\R^N)$, uniformly with respect to $n$, and that the sequence $(u_n-\sum_{i\in\N}\rho_n^i(\phi_i))_{n\in\N}$ converges to zero in $L(p^*,q)$ for every index $q>p$.

Moreover we have that, for any pair of indexes $i$ and $j$, the two corresponding sequences of rescalings $(\rho_n^i)_{n\in\N}$ and $(\rho_n^j)_{n\in\N}$ are mutually diverging, that
\begin{equation}\label{solimini2}
\sum_{i=1}^{+\infty}\nor{\phi_i}_{1,p}^p\leq M\ ,
\end{equation}
where $M$ is the limit of $(\nor{u_n}_{1,p}^p)_{n\in\N}$, and that the sequence \\
$(u_n-\sum_{i\in\N}\rho_n^i(\phi_i))_{n\in\N}$ converges to zero in $H^{1,p}(\R^N)$ if and only if (\ref{solimini2}) is an equality.
\end{thm}

Now we can state the result

\begin{thm}\label{sa<s}
 If $S_{A,a}<S$ then $S_{A,a}$ is attained.
\end{thm}
\begin{proof}
Let us consider a minimizing sequence $u_n\in\dom$ to $S_{A,a}$. In
particular, it is bounded in $\dom$. By Solimini's theorem
(\ref{solimini}), up to subsequences, there will exist a sequence
$\phi_i\in\dom$ and a sequence of mutually divergent rescalings
$\rho_n^i$ defined as
$\rho_n^i(u)=(\lambda_n^i)^{\frac{N-2}{2}}u(x_n+\lambda_n^i(x-x_n))$,
such that $\sum_{i}\rho_n^i\phi_i\in\dom$ and $u_n -
\sum_{i}\rho_n^i\phi_i\rightarrow0$ in $L^{2^*}$. In general the
rescalings may be mutually divergent by dilation (concentration or
vanishing) or by translation. In our case the Rayleigh quotient is
invariant under dilations, so the rescalings' divergence by dilation
cannot occur. By that, we mean the possibility to normalize the
modula $\lambda_n^1=1$ for each $n$.

Moreover, there exists at least a nontrivial function $\phi_i$,
namely $\phi_1$, which we choose to denote just $\phi$, in such a
way that we can write $u_n(\cdot) - \phi(\cdot+x_n) -
\sum_{i\geq2}\rho_n^i\phi_i\rightarrow0$ in $L^{2^*}$ with a little
abuse of notation, meaning that
$(\rho_n^1)^{-1}u_n-\phi(\cdot+x_n)-\sum_{i\geq2}(\rho_n^1)^{-1}\rho_n^i\phi_i\rightarrow0$
in $L^{2^*}$, where $(\rho_n^1)^{-1}u_n$ is again a minimizing
sequence and
$\sum_{i\geq2}(\rho_n^1)^{-1}\rho_n^i\phi_i\rightarrow0$ a.e. in
$\R^N$, because of the rescalings' mutual divergence. We can think
the sequence $u_n$ is normalized in $L^{2^*}$-norm. In this way the
sequence $\sum_{i\geq2}\rho_n^i\phi_i$ is also equibounded in
$L^{2^*}$; so that $\sum_{i\geq2}\rho_n^i\phi_i\rightharpoonup0$ in
$L^{2^*}$. At the same time even
$\sum_{i\geq2}\rho_n^i\phi_i\rightharpoonup0$ in $\dom$.

If we call for a moment $v_n=\sum_{i\geq2}\rho_n^i\phi_i$, we have
$$\abs{\abs{v_n+\phi}^{2^*}-\abs{\phi}^{2^*}-\abs{v_n}^{2^*}}\leq C\left(\abs{v_n}^{2^*-1}\abs{\phi}+\abs{v_n}\abs{\phi}^{2^*-1}\right)$$
from which
$$\int_{\R^N}\abs{u_n}^{2^*}=\int_{\R^N}\abs{\phi}^{2^*}+\int_{\R^N}\abs{v_n}^{2^*}+o(1)$$
thanks to the weak convergence $v_n\rightharpoonup0$ in $L^{2^*}$.
At the same time
\begin{eqnarray*}
\int_{\R^N}\abs{\nabla_A\left(\phi(\cdot+x_n)+v_n\right)}^2&=&\int_{\R^N}\abs{\nabla_A\phi(\cdot+x_n)}^2+\int_{\R^N}\abs{\nabla_A v_n}^2 \\
&&+2\int_{\R^N}\nabla_A\phi(\cdot+x_n)\cdot\nabla_A v_n\\
&=& \int_{\R^N}\abs{\nabla_A\phi(\cdot+x_n)}^2+\int_{\R^N}\abs{\nabla_A v_n}^2 +o(1)
\end{eqnarray*}
thanks to the weak convergence $v_n\rightharpoonup0$ in $\dom$.
So that
\begin{eqnarray*}
 S_{A,a}&\leftarrow&\frac{\displaystyle\int_{\R^N}\abs{\nabla_A\phi(\cdot+x_n)}^2+\int_{\R^N}\abs{\nabla_Av_n}^2+o(1)}{\displaystyle\left(\int_{\R^N}\abs{\phi}^{2^*}+\int_{\R^N}\abs{v_n}^{2^*}+o(1)\right)^{2/2^*}}\\
&\geq&S_{A,a}\frac{\displaystyle\left(\int_{\R^N}\abs{\phi}^{2^*}\right)^{2/2^*}+\left(\int_{\R^N}\abs{v_n}^{2^*}\right)^{2/2^*}+o(1)}{\displaystyle\left(\int_{\R^N}\abs{\phi}^{2^*}+\int_{\R^N}\abs{v_n}^{2^*}+o(1)\right)^{2/2^*}}\ ,
\end{eqnarray*}
and in order not to fall in contradiction 
the previous coefficient must tend to zero, and then $\int_{\R^N}\abs{v_n}^{2^*}\rightarrow0$.

In conclusion, we have $\nor{\sum_{i\geq2}\rho_n^i\phi_i}_{2^*}\rightarrow0$ and therefore the strong $\dom$ convergence $u_n(\cdot) -
\phi(\cdot+x_n)\rightarrow0$, since we have an equality in (\ref{solimini2}) in Theorem (\ref{solimini}).

At this point we aim to exclude the rescalings' translation
divergence. Let us suppose by contradiction this occurs: Lemma (\ref{lemmadivergenzatraslazione}) proves that if
$\abs{x_n}\rightarrow+\infty$, then the quotient evalueted on the
minimizing sequence $\Phi(\cdot+x_n)$ tends to
$\dfrac{\int_{\R^N}\abs{\nabla\phi}^2}{(\int_{\R^N}\abs{\phi}^{2^*})^{2/2^*}}$
which is greater equivalent than $S$, so we have a contradiction. \qed
\end{proof}

\subsection{The case $a\leq0$}

In order to investigate when the infimum is attained depending on the magnetic vector potential $A$ and the electric potential $a$, we start from the simplest cases. The first of them is the case $a\leq0$.

\begin{prop}
If $a\leq0$, $S_{A,a}$ is not achieved.
\end{prop}
\begin{proof}
First of all, in this case we have $S_{A,a}=S$. Indeed, by
diamagnetic inequality, we have
$$\qa -\int_{\R^{N}}\frac{a(\theta)}{\abs{x}^2}\,u^2 \geq \int_{\R^N}\abs{\nabla \abs{u}}^2 - \int_{\R^{N}}\frac{a(\theta)}{\abs{x}^2}\,u^2$$
from which we have $S_{A,a}\geq S$.

Suppose by contradiction $S_A$ is achieved on a function $\phi$.
Following the previous argument by Solimini's theorem, according to
the negativity of the electric potential, we get $S_{A,a}\geq S+c$,
where $c$ is a positive constant due to the convergence of the term
$\dfrac{\int_{\R^{N}}\frac{a(\theta)}{\abs{x}^2}\,\phi(\cdot+x_n)^2}{(\int_{\R^N}\abs{\phi(\cdot+x_n)}^{2^*})^{2/2^*}}$.
So we get $S_{A,a}>S$, a contradiction.

Note here we used the considerable fact that
$$\inf_{u\in\dom\setminus\{0\}}\frac{\displaystyle\int_{\R^N}\abs{\nabla \abs{u}}^2}{\displaystyle\left(\int_{\R^N}\abs{u}^{2^*}\right)^{2/2^*}} = S\ .$$
Its proof is based on the idea that $S$ is achieved over a radial function. \qed
\end{proof}

\subsection{The case $a=0$ and $A\neq0$}

In this case we expect in general the infimum is not achieved.
Indeed, first of all we have $S_{A,a}=S$, because we have already
seen in general $S_{A,a}\leq S$, and in this case the diamagnetic
inequality gives the reverse inequality. There is a simple case in
which we can immediately deduce a result.
\begin{rem}
 If the vector potential $\dfrac{A}{\abs{x}}$ is a gradient of a function $\Theta\in L^1_{loc}(\R^N)$ such that $\nabla\Theta\in L^{N,\,\infty}(\R^N)$, then $S_{A,a}$ is achieved. \\
Indeed, suppose $\dfrac{A}{\abs{x}}=\nabla\Theta$ for a function $\Theta\in L^1_{loc}(\R^N)$ such that its gradient has the regularity mentioned above. The change of gauge $u\mapsto \e^{+i\Theta}u$ makes the problem (\ref{sa}) equivalent to (\ref{s}), so that the infimum is necessarly achieved.\\
Just a few words about the regularity asked to $\nabla\Theta$. In
order to have the minimum problem wellposed, it would be sufficient
$\nabla\Theta\in L^2$. But if we require the function
$\e^{-i\Theta}u\in\dom$ for any $u\in\dom$, this regularity is not
sufficient any more. Rather, everything works if $\nabla\Theta\in
L^{N,\,\infty}(\R^N)$.
\end{rem}

Now, suppose the infimum $S_{A,a}=S$ is achieved on a function $u\in\dom$. Then we have
$$S=\frac{\displaystyle\qa}{\bigg(\displaystyle\int_{\R^N}\abs{u}^{2^*}\bigg)^{2/2^*}}
\geq
\frac{\displaystyle\int_{\R^N}\abs{\nabla\abs{u}}^2}{\displaystyle\bigg(\int_{\R^N}\abs{u}^{2^*}\bigg)^{2/2^*}}\geq
S \ .$$ So it is clear the equality must hold in the diamagnetic
inequality in order not to fall into a contradiction. We have the
following chain of relations:
\begin{eqnarray*}
 \abs{\nabla\abs{u}}=\abs{Re\Big(\frac{\overline{u}}{\abs{u}}\nabla u\Big)}=\abs{Im\Big(i\frac{\overline{u}}{\abs{u}}\nabla u\Big)}=\abs{Im\Big(i\nabla u-\frac{A}{\abs{x}}u\Big)\frac{\overline{u}}{\abs{u}}}
\leq\abs{\Big(i\nabla
u-\frac{A}{\abs{x}}u\Big)\frac{\overline{u}}{\abs{u}}}\ .
\end{eqnarray*}
In order that the equality holds in the last line
\mbox{$Re\Big\{\Big(i\nabla u-\frac{A}{\abs{x}}u\Big)\overline
u\Big\}$} must vanish. Expanding the expression one finds the equivalent
condition is \mbox{$\frac{A}{\abs{x}}=Re\Big(i\frac{\nabla
u}{u}\Big)$}. We can rewrite $i\frac{\nabla u}{u}=i\frac{\nabla
u}{\abs{u}^2}\overline{u}$ and see
$$
Re\Big(i\frac{\nabla
u}{u}\Big)=\frac{-Re(u)\nabla\left(Im(u)\right)+Im(u)\nabla\left(Re(u)\right)}{\abs{u}^2}=-\nabla\left(\arctan\frac{Im(u)}{Re(u)}\right)
$$
which is equivalent to $-\frac{A}{\abs{x}}=\nabla\Theta$ where
$\Theta$ is the phase of $u$.

In conclusion, we can resume our first remark both with this argument to state the following
\begin{prop}
If the electric potential $a=0$, the infimum $S_{A,a}$ is \\achieved if and only if $\dfrac{A}{\abs{x}}=\nabla\Theta$. In this case $\Theta$ is the phase of the minimizing function.
\end{prop}

\subsection{The general case: sufficient conditions}
In Theorem (\ref{sa<s}) we proved that a sufficient condition for
the infimum achieved is $S_{A,a}<S$. In this section we look for the
hypotheses on $A$ or $a$ which guarantee this condition.

\begin{prop} \label{teocondsuff1}
Suppose there exist a small ball ${B_{\delta}(x_0)}$ centered in $x_0\in\sphere{N-1}$ in which
$$
 a(x)-\abs{A(x)}^2\geq\lambda>0 \quad \textrm{a.e. $x\in B_{\delta}(x_0)$.}
$$
Then $S_{A,a}<S$ and so $S_{A,a}$ is achieved.
\end{prop}
\begin{proof}
We define
$$
 \mathcal{H}_A(\Omega)=\overline{\cinf{\Omega}}^{(\int_{\Omega}\abs{\nabla_A u}^2)^{1/2}}
$$
 the closure of compact supported functions with respect to the norm associated to the quadratic form.
We have the following chain of relations:
\begin{eqnarray*}
 S_{A,a} &\leq& \inf_{u\in\mathcal{H}_A(B_{\delta}(x_0))\setminus\{0\}}\frac{\displaystyle\int_{\R^N}\abs{\nabla_Au}^2 - \displaystyle\int_{\R^N}\frac{a}{\abs{x}^2}\abs{u}^2}{\displaystyle\left( \int_{\R^N}\abs{u}^{2^*}\right)^{2/2^*}}\\
&\leq& \inf_{u\in\mathcal{H}_A(B_{\delta}(x_0),\R)\setminus\{0\}}\frac{\displaystyle\int_{B_{\delta}(x_0)}\abs{\nabla_Au}^2 - \displaystyle\int_{B_{\delta}(x_0)}\frac{a}{\abs{x}^2}u^2}{\displaystyle\left( \int_{B_{\delta}(x_0)}\abs{u}^{2^*}\right)^{2/2^*}}
\end{eqnarray*}
since the quotient is invariant under Solimini's rescalings and we restrict 
to a proper subset of functions. 
When we check the quotient over a real function, it reduces to
$$
 \frac{\displaystyle\int_{B_{\delta}(x_0)}\abs{\nabla u}^2 - \displaystyle\int_{B_{\delta}(x_0)}\frac{\abs{A}^2-a}{\abs{x}^2}u^2}{\displaystyle\left( \int_{B_{\delta}(x_0)}\abs{u}^{2^*}\right)^{2/2^*}}\,,
$$
so the thesis follows from \cite{BN83}, Lemma (1.1). \qed
\end{proof}

\smallskip

\begin{rem}
 We can resume the results reached until now: in case the magnetic vector potential 
$\frac{A}{\abs{x}}$ is a gradient, the infimum $S_{A,a}$ is achieved if $a\equiv0$ 
or if its essential infimum is positive and sufficiently small in a neighborhood far 
from the origin (we mean $\nor{a}_\infty\leq(N-2)^2/4$ in order to keep the quadratic 
form positive definite); while it is never achieved provided $a\leq0$, neither in case 
the magnetic potential is a gradient, nor in case it is not. On the other hand, in order 
to have $S_{A,a}$ achieved, if the magnetic vector potential is not a gradient we need to 
assume it has a suitably low essential supremum somewhere in a ball far from the origin 
in relation to the electric potential $a$ (see Proposition (\ref{teocondsuff1})).
\end{rem}


Anyway, it seems reasonable what is important here is not the essential
supremum of $\frac{A}{\abs{x}}$ (or $A$, since we play far from the
origin), but "the distance" between the magnetic vector potential
and the set of gradients. Pursuing this idea, it seems possible to
interpretate a suitable (to be specified) norm of
$\curl{\frac{A}{\abs{x}}}$ as a measure of this distance. In order
to specify these ideas we refer to \cite{Lein83} and \cite{BB07}. We
recall the following
\begin{defn} \cite{Lein83}
 Let $\Omega$ be a open set of $\R^N$ and $\vett{a}\,,\vett{b}\in L^1_{loc}(\Omega)$.
 We say that $\vett{a}$ and $\vett{b}$ are related by a gauge transformation, $\vett{a}\sim_{\Omega}\vett{b}$,
 if there is a distribution $\lambda\in\textit{D}'(\Omega)$ satisfying $\vett{b}=\vett{a}+\nabla\lambda$.

By $\curl{\vett{a}}$ we denote the skew-symmetric, matrix-valued distribution having $\partial_i\vett{a}_j-\partial_j\vett{a}_i\in\textit{D}'(\Omega)$ as matrix elements.
\end{defn}

\begin{lem} \label{leinfelder} \cite{Lein83}
 Let $\Omega$ be any open subset of $\R^N$, $1\leq p<+\infty$ and $\vett{a}\,,\vett{b}\in L^p_{loc}(\Omega)$. Then every $\lambda$ satisfying $\vett{b}=\vett{a}+\nabla\lambda$ belongs to $W^{1\,,p}(\Omega)$. If $\Omega$ is simply-connected then $$\vett{a}\sim_{\Omega}\vett{b}\ \Longleftrightarrow\ \curl{\vett{a}}=\curl{\vett{b}}\ .$$
\end{lem}

\begin{thm} \label{bourgain} \cite{BB07}
 Let $M=(0,1)^N$ be the N-dimensional cube of $\R^N$ with $N\geq2$ and $1\leq l\leq N-1$. Given any $X$ a $l$-form with coefficients in $W^{1,N}(M)$ there exists some $Y$ a $l$-form with coefficients in $W^{1,N}\cap L^{\infty}(M)$ such that $$dY=dX$$ and $$\nor{\nabla Y}_N+\nor{Y}_{\infty}\leq C\nor{dX}_N\ .$$
\end{thm}

The Theorem (\ref{bourgain}) will be very useful in our case
choosing $l=1$, so that the external derivative is the curl of the
vector field which represents the given 1-form.

Suppose $\frac{A}{\abs{x}}\in W^{1,N}(B_{\delta}(x_0))$ in a ball
far from the origin. Then for Theorem (\ref{bourgain}) there exists
a vector field $Y\in L^{\infty}\cap W^{1,N}(B_{\delta}(x_0))$ such
that $\curl{\frac{A}{\abs{x}}}=\curl{Y}$ and $\nor{Y}_{\infty}\leq
C\nor{\curl{\frac{A}{\abs{x}}}}_N$. By Lemma (\ref{leinfelder}), $Y$
is related to $\frac{A}{\abs{x}}$ by a gauge transformation, so, in
the spirit of Theorem (\ref{teocondsuff1}), it is sufficient
$\nor{\curl{\frac{A}{\abs{x}}}}_N$ is not too large in order to have
$S_{A,a}<S$ and hence $S_{A,a}$ achieved.

\section{Symmetry of solutions}

We recall once again in general $S_{A,a}\leq S$. When $S_{A,a}=S$
and $Q_{A,a}(u)>Q(u)$ for any $u\in\dom$, e.g. when $a\leq0$ but not
identically 0, we lose compactness since clearly $S_{A,a}$ cannot be
attained. In this section we follow the idea that introducing
symmetry properties to the quadratic form can help in growing the
upper bound for $S_{A,a}$, in order to increase the probability for
it to be achieved.

We basically follow the ideas in \cite{Ter96}, assuming the
dimension $N\geq4$.

Let us write $\R^N=\R^2\times\R^{N-2}$ and denote $x=(z,y)$. Let us
fix $k\in\N$, and suppose there is a $\Z_k\times SO(N-2)$
group-action on $\dom$, denoting
\begin{eqnarray*}
\domk=\{u(z,y)\in\dom\
\textrm{s.t.}\ u(\e^{i\frac{2\pi}{k}}z,Ry)=u(z,y)\ \textrm{for any
$R\in SO(N-2)$}\}
\end{eqnarray*}
the fixed point space. In order to have the
quadratic form invariant under this action, let us suppose that
$a(\theta)\equiv a\in\R^{-}$ and
\begin{equation}\label{A-Zksymm}
A\left(\frac{\e^{i\frac{2\pi}{k}}z,y}{\abs{(z,y)}}\right)=\left(\e^{i\frac{2\pi}{k}}(A_1,A_2),A_3\right)\left(\frac{z,\abs{y}}{\abs{(z,y)}}\right)\ .
\end{equation}
We set the problem
\begin{equation}\label{ska}
S^{k}_{A,a}:=\inf_{u\in\domk}\frac{Q_{A,a}(u)}{\nor{u}^2_{2^*}}\ .
\end{equation}

\begin{thm}\label{s^kachieved}
 If $S^{k}_{A,a}<k^{2/N}S$ then it is achieved.
\end{thm}
\begin{proof}
 Let us consider a minimizing sequence $\{u_n\}$. The space $\domk$ is a close subspace in $\dom$, so Solimini's Theorem (\ref{solimini}) holds in $\domk$.
Up to subsequences we can find a sequence $\Phi_i\in\domk$ and a
sequence of mutually diverging rescalings $\rho_n^i$ such that $u_n
- \sum_i \rho_n^i(\Phi_i) \rightarrow 0$ in $L^{2^*}$. Thanks to the
quotient's invariance under dilations, the rescalings should be
mutually diverging only by translations, as already pointed out.

We basically follow the proof of Theorem (\ref{sa<s}). Since
$\{u_n\}$ is a minimizing sequence for the Sobolev quotient, we can
claim there exists at least a function of the form
\begin{equation}\label{gobbe}
\sum_{j=1}^k \Phi(\cdot+x_n^j)
\end{equation}
 which is not zero, and we can assume its relative rescaling, namely $\rho_n^1$, is the identity.
 Of course there could exist more than $k$ points $x_n^j$ in the form (\ref{gobbe}), but certainly they are in number at least $k$,
 and the abovementioned form remains correct. The remaining part $\sum_{i\geq2} \rho_n^i(\Phi_i)$ weakly converges to zero in $L^{2^*}$,
 and also their $L^{2^*}$-norms converge to zero.
 Thus $u_n - \sum_{j=1}^k \Phi(\cdot+x_n^j)\rightarrow 0$ in $L^{2^*}$, and for the last assertion of Solimini's theorem,
 we gain the strong convergence $u_n - \sum_{j=1}^k \Phi(\cdot+x_n^j)\rightarrow 0$ in $\domk$.

Arguing by contradiction, let us suppose $S^{k}_{A,a}$ is not
achieved. This means $\abs{x_n^j}\rightarrow+\infty$ as
$n\rightarrow\infty$ for all $j=1,\dots, k$ (the symmetry must be
preserved). Evaluating the quotient over $u_n$ is the same as
evaluating it over $\sum_{j=1}^k \Phi(\cdot+x_n^j)$ up to $o(1)$.
According to the assumption $\abs{x_n^j}\rightarrow+\infty$ we
compute
\begin{equation} \label{sus1}
 \int_{\R^N}\Big\vert\nabla_A \Big(\sum_{j=1}^k \Phi(\cdot+x_n^j)\Big)\Big\vert^2\ 
-\ \int_{\R^N}\frac{a}{\abs{x}^2}\Big\vert\sum_{j=1}^k \Phi(\cdot+x_n^j)\Big\vert^2\ =\ k\int_{\R^N}\abs{\nabla\Phi}^2+o(1).
\end{equation}

We claim that
\begin{equation} \label{sus2}
\int_{\R^N}\Big\vert\sum_{j=1}^k \Phi(\cdot+x_n^j)\Big\vert^{2^*}=k\int_{\R^N}\abs{\Phi}^{2^*}\,+\,o(1)\ .
\end{equation}
A proof of this fact is based on the inequality $\abs{\abs{a+b}^{2^*}-\abs{a}^{2^*}-\abs{b}^{2^*}}\leq C\big(\abs{a}^{2^*-1}\abs{b}+\abs{a}\abs{b}^{2^*-1}\big)$ applied $(k-1)$ times. As in the previous equivalence (\ref{sus1}) the mixed terms are $o(1)$ because of the divergence $\abs{x_n^j}\rightarrow+\infty$.

Using (\ref{sus1}) and (\ref{sus2}) we can write
$$
 S^k_{A,a}=\lim_{n\rightarrow+\infty}\frac{Q_{A,a}\big(\sum_{j=1}^k\Phi(\cdot+x_n^j)\big)}
{\displaystyle\bigg(\int_{\R^N}\Big\vert\sum_{j=1}^k\Phi(\cdot+x_n^j)\Big\vert^{2^*}\bigg)^{2/2^*}}
=\frac{\displaystyle k\int_{\R^N}\abs{\nabla\Phi}^2}{\displaystyle\bigg(k\int_{\R^N}\abs{\Phi}^{2^*}\bigg)^{2/{2^*}}}
\geq k^{2/N}\,S,
$$
a contradiction. \qed
\end{proof}

We emphasize under these assumptions the minimum has the form $\overline{u}=\sum_{j=1}^k\Phi(\cdot+x^j)$.

\begin{rem}
 The above result is actually a symmetry breaking result for the equation associated to these minimum problems. Indeed, let us consider the equation
\begin{equation}\label{eq}
 -\Delta_A u = \frac{a}{\abs{x}^2}+\abs{u}^{2^*-2}u \qquad \textrm{in $\R^N$,}
\end{equation}
where $-\Delta_A$ denotes the differential operator we have called
\emph{magnetic Laplacian}. Then the minima of (\ref{sa}) are
solutions to (\ref{eq}) and so are those of (\ref{ska}), thanks to
the \emph{Symmetric Criticality Principle} (the quotient is
invariant under the $\Z_k\times SO(N-2)$ group-action). Thus, when
the electric potential is constant and negative, we find a
multiplicity of solutions to (\ref{eq}) depending on $k$ (we would
say an infinite number, at least for $k$ not multiples to each
other), and each of them is invariant under rotations of angle
$2\pi/k$, respectively.
\end{rem}

Now we want to check whenever the condition $S^k_{A,a}<k^{2/N}S$ is
fulfilled. Let us pick $k$ points in $\R^N\setminus\{0\}$ of the
form $x_j=(R\e^{\frac{2\pi i}{k}j}\xi_0,0)$ where $\abs{\xi_0}=1$,
and denote
\begin{equation}\label{wj}
w_j=\e^{\frac{2\pi
i}{k}jm}\frac{\big(N(N-2)\big)^{\frac{N-2}{4}}}{\big(1+\abs{x-x_j}^2\big)^{\frac{N-2}{2}}}\
.
\end{equation}
In this way the sum $\sum_{j=1}^{k}w_j$ is an element of $\domk$.
Additionally we notice $w_j$ are minimizers of the usual Sobolev
quotient, and they satisfy
\begin{equation}\label{eqwj}
-\Delta w_j=\abs{w_j}^{2^*-2}w_j \qquad \textrm{in $\R^N$.}
\end{equation}
It is worth to notice that both
$$
\frac{\displaystyle\int_{\R^N}\abs{\nabla
w_j}^2}{\displaystyle\left(\int_{\R^N}\abs{w_j}^{2^*}\right)^{2/2^*}}=S
$$
and (\ref{eqwj}) imply
\begin{equation}\label{wj-S}
\int_{\R^N}\abs{\nabla w_j}^2=\int_{\R^N}\abs{w_j}^{2^*}=S^{N/2}\,.
\end{equation}

We state the following
\begin{prop}\label{valangadiconti}
 Choosing $R$ big enough, the quotient evaluated over \\$\sum_{j=1}^{k}w_j$ 
is strictly less than $k^{2/N}S$, and so the infimum $S^k_{A,a}$.
\end{prop}

In order to prove it, we need some technical results, whose proofs are postponed to the next subsection. We basically
follow the ideas in \cite{Ter96}.

For seek of semplicity, we introduce the following notation:
\begin{eqnarray*}
\alpha&=&\int_{\R^N}Re\Big\{\sum_{i\neq j}\abs{w_i}^{2^*-2}w_i\,\overline{w_j}\Big\} \\
\beta&=&\int_{\R^N}\frac{\abs{A}^2-a}{\abs{x}^2}\Big\vert\sum_{j=1}^{k}w_j\Big\vert^2\\
\gamma&=&Re\Big\{i\int_{\R^N}\frac{A}{\abs{x}}\cdot\sum_{i,j}\nabla
w_i \overline{w_j}\Big\}\ .
\end{eqnarray*}

\begin{lem}\label{lemma positivita'alfa}
 It holds $\alpha\geq0$.
\end{lem}

\begin{lem}\label{lemma disuguaglianza den}
For every positive $\delta$ there exists a positive constant $K_\delta$ 
(independent of $k$) such that if 
$$\dfrac{\abs{x_i-x_j}^2}{\log\abs{x_i-x_j}}\geq K_\delta\,(k-1)^{2/(N-2)} \quad \forall i\neq j$$
then
\begin{equation}\label{disuguaglianza den}
\int_{\R^N} \bigg\vert\sum_{j=1}^k w_j\bigg\vert^{2^*}\ \geq \ k\,S^{N/2} +
2^*(1-\delta) \int_{\R^N}Re\Big\{\sum_{i\neq
j}\abs{w_i}^{2^*-2}w_i\,\overline{w_j}\Big\}.
\end{equation}
\end{lem}

\begin{lem}\label{scelta alfa,beta,gamma}
Given Lemma (\ref{lemma disuguaglianza den}), it is possible to choose $R$ and $k$ in such a way that the quantity 
$$ 1 + \frac{1}{kS^{N/2}}\Big\{\beta-2\gamma+\alpha\Big(-1+\delta+\frac{2-\delta}{kS^{N/2}}(2\gamma-\beta)\Big)\Big\} $$
is positive and strictly less than 1.
\end{lem}

\emph{\textbf{Proof of Proposition (\ref{valangadiconti}).}} Let us evaluate the quotient over \\ $\sum_{j=1}^{k}w_j$\,:
\begin{eqnarray}\label{sus3}
\lefteqn{\int_{\R^N}\Big\vert\nabla_A \big(\sum_{j=1}^{k}w_j\big)\Big\vert^2 - \int_{\R^N}\frac{a}{\abs{x}^2}\Big\vert\sum_{j=1}^{k}w_j\Big\vert^2\nonumber}\\
&&=\int_{\R^N}\bigg\{\sum_{j=1}^{k}\abs{\nabla w_j}^2 + Re\Big\{\sum_{i\neq j}\nabla w_i\cdot\nabla \overline{w_j}\Big\} + \frac{\abs{A}^2-a}{\abs{x}^2}\Big\vert\sum_{j=1}^{k}w_j\Big\vert^2 \nonumber\\
&&\indent- 2\,Re\Big\{i\frac{A}{\abs{x}}\cdot\sum_{i,j}\nabla w_i \overline{w_j}\Big\}\bigg\}\nonumber \\
&&=kS^{N/2} + \int_{\R^N}\bigg\{Re\Big\{\sum_{i\neq j}\abs{w_i}^{2^*-2}w_i\,\overline{w_j}\Big\} +
\frac{\abs{A}^2-a}{\abs{x}^2}\Big\vert\sum_{j=1}^{k}w_j\Big\vert^2 \nonumber\\
&& \indent- 2\,Re\Big\{i\frac{A}{\abs{x}}\cdot\sum_{i,j}\nabla w_i
\overline{w_j}\Big\}\bigg\}
\end{eqnarray}
where in the last equality we have used (\ref{wj-S}) and the
equation (\ref{eqwj}). Now we use Lemma (\ref{lemma disuguaglianza den}) which
states the lower bound (\ref{disuguaglianza den}) for the
denominator of our quotient.
Thus using (\ref{sus3}) and (\ref{disuguaglianza den}) the quotient is
\begin{eqnarray*}
\lefteqn{\frac{Q_{A,a}\big(\sum_{j=1}^{k}w_j\big)}{\nor{\big(\sum_{j=1}^{k}w_j\big)}_{2^*}^2}}\\
&&\leq\left(\displaystyle kS^{N/2} + \int_{\R^N}\bigg\{Re\Big\{\sum_{i\neq j}\abs{w_i}^{2^*-2}w_i\,\overline{w_j}\Big\} 
+ \frac{\abs{A}^2-a}{\abs{x}^2}\Big\vert\sum_{j=1}^{k}w_j\Big\vert^2 \right.\\
&&\indent\left.- 2\,Re\Big\{i\frac{A}{\abs{x}}\cdot\sum_{i,j}\nabla w_i \overline{w_j}\Big\}\bigg\}\right)\\
&& \indent \cdot\left(kS^{N/2}
+2^*(1-\delta/2)\int_{\R^N}Re\Big\{\sum_{i\neq j}\abs{w_i}^{2^*-2}w_i\,\overline{w_j}\Big\}\right)^{-2/{2^*}}\\
&&=k^{2/N}S\,\bigg(1+\frac{1}{kS^{N/2}}(\alpha+\beta-2\gamma)\bigg)\bigg(1-\frac{2(1-\delta/2)}{kS^{N/2}}\alpha\bigg)+o(1)
\end{eqnarray*}
where in the last line we have expanded the denominator in Taylor's 
serie since the argument is very close to zero if $R$ is large.
Up to infinitesimal terms of higher order, the coefficient of $k^{2/N}S$ is
$$1+\frac{1}{kS^{N/2}}\beta-\frac{2}{kS^{N/2}}\gamma
+\frac{1}{kS^{N/2}}\alpha\big(-1+\delta+\frac{2-\delta}{kS^{N/2}}(2\gamma-\beta)\big)\ .$$
Now we invoke Lemma (\ref{scelta alfa,beta,gamma}) to conclude the proof. \qed

\subsection{Proofs of technical lemmas}

In order to prove Lemmas (\ref{lemma positivita'alfa}), (\ref{lemma disuguaglianza den}) 
and (\ref{scelta alfa,beta,gamma}) we need supplementary results mainly about asymptotics 
of the quantities involved.

\begin{lem}\label{lemma_asintotici}
We have, as $\abs{x_i-x_j}\rightarrow+\infty$ and
$\abs{x_i}\rightarrow+\infty$
\begin{eqnarray}
\int_{\R^N} \abs{w_i}^{2^*-2}w_i\,\overline{w_j} &=& O\big(\frac{1}{\abs{x_i-x_j}^{N-2}}\big) \label{asintotico_alfa}\\ 
\int_{\R^N} \abs{w_i\,\overline{w_j}}^{2^*/2} &=& O\big(\frac{\log{\abs{x_i-x_j}}}{\abs{x_i-x_j}^N}\big) \label{secondo_asintotico}\\
\int_{\R^N} \frac{\abs{w_j}^2}{\abs{x}^2} &=&
\left\{\begin{array}{ll}
 O\big(\displaystyle\frac{\log R}{R^2}\big) \quad \textrm{if $N=4$}\\
 O\big(\displaystyle\frac{1}{R^2}\big) \quad \textrm{for $N\geq5$}
 \end{array}\right. \label{asintotico_beta}\\ 
\int_{\R^N} \frac{1}{\abs{x}} \cdot \abs{\nabla
w_j}\,\abs{w_i}  &=& O\big(\frac{1}{R\,\abs{x_i-x_j}^{N-3}}\big).
\label{asintotico_gamma}
\end{eqnarray}
\end{lem}
\begin{proof}
For what concerns (\ref{asintotico_alfa}), (\ref{secondo_asintotico}) and
(\ref{asintotico_beta}) we refer to \cite{Ter96}.


About (\ref{asintotico_gamma}) we have
\begin{eqnarray*}
\int_{B_{R/2}(0)}\frac{1}{\abs{x}}\abs{\nabla w_j}\,\abs{w_i} 
= O(\frac{1}{R^{N-1}})\,O(\frac{1}{R^{N-2}})\,O(R^{N-1})=O(\frac{1}{R^{N-2}})
\end{eqnarray*}
since in $B_{R/2}(0)$ $\abs{x-x_i}\geq R-\abs{x}\geq R/2$ and the
same holds for $\abs{x-x_j}$;
\begin{eqnarray*}
\int_{B_{\vert x_i-x_j\vert/4}(x_i)}\frac{1}{\abs{x}}\abs{\nabla w_j}\,\abs{w_i} 
&=& O(\frac{1}{R})\,O(\frac{1}{\abs{x_i-x_j}^{N-1}})\int_{B_{\vert x_i-x_j\vert/4}(x_i)}\abs{w_i}\\
&=& O(\frac{1}{R\abs{x_i-x_j}^{N-3}})
\end{eqnarray*}
since in $B_{\vert x_i-x_j\vert/4}(x_i)$ $\abs{x}\geq \abs{x_i}-\abs{x-x_i}\geq R/2$ and
$\abs{x-x_j}\geq\abs{x_i-x_j}-\abs{x-x_i}\geq\frac{3}{4}\abs{x_i-x_j}$;
\begin{eqnarray*}
\int_{B_{\vert x_i-x_j\vert/4}(x_j)}\frac{1}{\abs{x}}\abs{\nabla w_j}\,\abs{w_i} 
&=& O(\frac{1}{R})O(\frac{1}{\abs{x_i-x_j}^{N-2}})\int_{B_{\vert x_i-x_j\vert/4}(x_j)}\abs{\nabla w_j}\\
&=& O(\frac{1}{R\abs{x_i-x_j}^{N-3}})
\end{eqnarray*}
since in $B_{\abs{x_i-x_j}/4}(x_j)$ $\abs{x}\geq \abs{x_j}-\abs{x-x_j}\geq R/2$ and
$\abs{x-x_i}\geq\abs{x_i-x_j}-\abs{x-x_j}\geq\frac{3}{4}\abs{x_i-x_j}$;
while in $\R^N\setminus (B_{R/2}(0)\cup B_{\abs{x_i-x_j}/4}(x_i)\cup
B_{\abs{x_i-x_j}/4}(x_j))$ we have $\abs{x}\geq R/2$, and via H\"{o}lder inequality
\begin{eqnarray*}
\lefteqn{\int_{\R^N\setminus (B_{R/2}(0)\cup B_{\vert x_i-x_j\vert/4}(x_i)\cup
B_{\vert x_i-x_j\vert/4}(x_j))}\frac{1}{\abs{x}}\abs{\nabla w_j}\,\abs{w_i}}\\ 
&&=\left\{
\begin{array}{ll}
 O\big(\dfrac{1}{R\abs{x_i-x_j}^2}\log\abs{x_i-x_j}\big) \quad \textrm{if $N=4$}\\
 O\big(\dfrac{1}{R\abs{x_i-x_j}^{2N-6}}\big) \quad \textrm{if $N\geq 5$.}
\end{array}\right.
\end{eqnarray*}
\qed
\end{proof}

\begin{rem}\label{distanza centri}
 The above asymptotics in Lemma (\ref{lemma_asintotici}) come in terms of $k$ and $R$ as we note
\begin{eqnarray*}
\abs{x_i-x_j}^2&=&R^2\sin^2{\frac{2\pi}{k}(i-j)}+R^2\Big(1-\cos{\frac{2\pi}{k}(i-j)}\Big)^2\\
&\sim& \left\{
\begin{array}{ll}
\dfrac{R^2}{k^2} + \dfrac{R^2}{k^4} = O(\dfrac{R^2}{k^2}) \quad &\textrm{if $\abs{i-j} \ll k$}\\
R^2  &\textrm{otherwise}
\end{array}\right..
\end{eqnarray*}
According to the previous asymptotic, we note we have the worst estimates in Lemma (\ref{lemma_asintotici}) for $\abs{i-j} \ll k$,
that is for the centers $x_i$, $x_j$ quite near to each other.
\end{rem}

\begin{lem}\label{asintotico gamma totale}
The following asymptotic behavior holds for $k\rightarrow +\infty$ and $R\rightarrow +\infty$
\begin{eqnarray*}
 \abs{\int_{\R^N} Re\Big\{i\,\frac{A}{\abs{x}}\cdot\sum_{l,j}\nabla w_l\,\overline{w_j}\Big\}}
\leq \left\{
\begin{array}{ll}
 O\big(\dfrac{k^2}{R^2}\big) \quad \textrm{if $N=4$} \\
 O\big(\dfrac{k^2\log k}{R^3}\big) \quad \textrm{if $N=5$}\\
 O\big(\dfrac{k^{N-3}}{R^{N-2}}\big) \quad \textrm{for $N\geq 6$.}
\end{array}\right.
\end{eqnarray*}
\end{lem}
\begin{proof}
First of all we note if $l=j$ the quantity in the statement is zero.
Next,
\begin{eqnarray*}
 \lefteqn{\abs{\int_{\R^N} Re\Big\{i\,\frac{A}{\abs{x}}\cdot\sum_{l,j}\nabla w_l\,\overline{w_j}\Big\}}= \abs{\sum_{l\neq j} \sin\frac{2\pi}{k}m(l-j) \int_{\R^N} \frac{A}{\abs{x}}\cdot \nabla\abs{w_l}\,\abs{w_j}}}\\
&&\leq \frac{C}{R^{N-2}} \sum_{l\neq j} \frac{\abs{\sin\frac{2\pi}{k}m(l-j)}}{\big(1-\cos\frac{2\pi}{k}(l-j)\big)^{\frac{N-3}{2}}}
= \frac{C\,k}{R^{N-2}} \sum_{l=1}^{k-1} \frac{\abs{\sin\frac{2\pi}{k}ml}}{\big(1-\cos\frac{2\pi}{k}l\big)^{\frac{N-3}{2}}}\\
&&\leq \frac{C\,m\,k}{R^{N-2}} \sum_{l=1}^{k-1} \frac{l/k}{(l/k)^{N-3}} 
= \frac{C\,m\,k^{N-3}}{R^{N-2}} \left\{
\begin{array}{ll}
 k \quad \textrm{if $N=4$} \\
 \log k \quad \textrm{if $N=5$}\\
 O(1) \quad \textrm{for $N\geq 6$.}
\end{array}\right.
\end{eqnarray*} \qed
\end{proof}

We recall the following result proved in \cite{Ter96}:
\begin{lem}
Let $s_1,\ldots,s_k \geq 0$. For every positive $\delta$ there
exists a positive constant $K_\delta$ (independent of $k$) such that
if
$$\dfrac{\abs{x_i-x_j}^2}{\log\abs{x_i-x_j}}\geq K_\delta\,(k-1)^{2/(N-2)} \quad \forall i\neq j$$
then
\begin{equation}\label{disuguaglianza susanna}
\int_{\R^N} \Big(\sum_{i=1}^k s_i\Big)^{2^*}\geq
k\,S^{N/2}+2^*(1-\delta/2)\int_{\R^N}\sum_{i\neq j}s_i^{2^*-1}s_j
\end{equation}
\end{lem}

\smallskip

\emph{\textbf{Proof of Lemma (\ref{lemma positivita'alfa}).}}
We split the sum in two contributions: indexes for which $\cos\frac{2\pi}{k}(j-l)\geq0$ (we will call them \textit{j,l pos}),
and indexes for which $\cos\frac{2\pi}{k}(j-l)\leq0$ (we will call them \textit{j,l neg}).
We note in the first case, we have $\abs{x_j-x_l} \sim \frac{R}{k}$, whereas in the second case $\abs{x_j-x_l} \sim R$.
Then
\begin{eqnarray}\label{stima alpha da sotto}
 \int_{\R^N} Re \bigg\{ \sum_{l,j\ pos}\abs{w_j}^{2^*-2}w_j\,\overline{w_l} \bigg\}&\geq& \int_{\R^N} Re \bigg\{ \sum_{j}\abs{w_j}^{2^*-2}w_j\,\overline{w_{j+1}} \bigg\} \nonumber\\
&=& k \int_{\R^N} Re \bigg\{ \abs{w_2}^{2^*-2}w_2\,\overline{w_1} \bigg\} = O(\frac{k^{N-1}}{R^{N-2}})\,.
\end{eqnarray}
On the other hand,
$$
 \int_{\R^N} Re \bigg\{ \sum_{l,j\ neg}\abs{w_j}^{2^*-2}w_j\overline{w_l} \bigg\}
\leq k^2 \int_{\R^N} Re \bigg\{ \abs{w_l}^{2^*-2}w_l\,\overline{w_1} \bigg\} = O(\frac{k^2}{R^{N-2}});
$$
so that for $k$ large enough we have the thesis. \qed

\emph{\textbf{Proof of Lemma (\ref{lemma disuguaglianza den}).}}
By convexity of the function $(\cdot)^{2^*/2}$ we have
\begin{eqnarray}
&&\abs{\sum_{j=1}^k w_j}^{2^*} = \left(\abs{\sum_{j=1}^k
w_j}^2\right)^{2^*/2} = \left(\sum_{i,j=1}^k
Re\{w_i\,\overline{w_j}\}\right)^{2^*/2} \nonumber \\
&&=\left(\sum_{i,j}\abs{w_i}\,\abs{w_j} -
\sum_{i,j}\abs{w_i}\,\abs{w_j}\Big(1-\cos\big(\frac{2\pi}{k}m(i-j)\big)\Big)\right)^{2^*/2} \nonumber \\
&&\geq \left(\sum_{i,j}\abs{w_i}\,\abs{w_j}\right)^{2^*/2} - \frac{2^*}{2}\left(\sum_{i,j}\abs{w_i}\,\abs{w_j}\right)^{2^*/2
-1}\sum_{i,j}\abs{w_i}\,\abs{w_j}\Big(1-\cos\big(\frac{2\pi}{k}m(i-j)\big)\Big)\,.
\label{secondo termine}
\end{eqnarray}

For what concerns the first term
$\left(\sum_{i,j}\abs{w_i}\,\abs{w_j}\right)^{2^*/2} =
\left(\sum_{j}^k\abs{w_j}\right)^{2^*}$, we can apply directly
inequality (\ref{disuguaglianza susanna}) in order to have
\begin{equation}\label{disuguaglianza susanna 2}
\int_{\R^N}\left(\sum_{j=1}^k\abs{w_j}\right)^{2^*} \ \geq\
k\,S^{N/2}+2^*(1-\delta/2)\int_{\R^N}\sum_{i\neq
j}\abs{w_i}^{2^*-1}\abs{w_j}\,.
\end{equation}
We want to stress that
$$
\int_{\R^N}\sum_{i\neq j}\abs{w_i}^{2^*-1}\abs{w_j}\ \geq\
\int_{\R^N}\sum_{j=1}^k\abs{w_{j}}^{2^*-1}\abs{w_{j+1}} \ =\
k\,\int_{\R^N}\abs{w_{1}}^{2^*-1}\abs{w_{2}}
$$
(see also \cite{Ter96}, equation (6.22)), so that
\begin{equation}\label{asintotico_alfa 2}
\int_{\R^N}\sum_{i\neq j}\abs{w_i}^{2^*-1}\abs{w_j}\ \geq\
O\big(\frac{k^{N-1}}{R^{N-2}}\big)\,.
\end{equation}

Now we focus our attention on the integral of the second term in
(\ref{secondo termine}): via H\"{o}lder inequality we have
\begin{eqnarray*}
\lefteqn{\int_{\R^N}\left(\sum_{i,j}\abs{w_i}\,\abs{w_j}\right)^{2^*/2
-1}\sum_{i,j}\abs{w_i}\,\abs{w_j}\Big(1-\cos\big(\frac{2\pi}{k}m(i-j)\big)\Big)}\\
&&\leq
\left(\int_{\R^N}\bigg(\sum_{i,j}\abs{w_i}\,\abs{w_j}\bigg)^{2^*/2}\right)^{\frac{2^*-2}{2^*}}\cdot\left(\int_{\R^N}\bigg(\sum_{i,j}\abs{w_i}\,\abs{w_j}\Big(1-\cos\big(\frac{2\pi}{k}m(i-j)\big)\Big)\bigg)^{2^*/2}\right)^{2/2^*}
\end{eqnarray*}
and
$$
\left(\int_{\R^N}\bigg(\sum_{i,j}\abs{w_i}\,\abs{w_j}\bigg)^{2^*/2}\right)^{\frac{2^*-2}{2^*}}
=
\left(\int_{\R^N}\bigg(\sum_{j}\abs{w_j}\bigg)^{2^*}\right)^{\frac{2^*-2}{2^*}}
\sim (k\,S^{N/2})^{\frac{2^*-2}{2^*}}
$$
thanks to inequality (\ref{disuguaglianza susanna}) and Lemma
(\ref{lemma_asintotici}). On the other hand
\begin{eqnarray*}
\lefteqn{\left(\int_{\R^N}\bigg(\sum_{i,j}\abs{w_i}\,\abs{w_j}\Big(1-\cos\big(\frac{2\pi}{k}m(i-j)\big)\Big)\bigg)^{2^*/2}\right)^{2/2^*}}\\
&& \leq \sum_{i,j}
\left(\int_{\R^N}\bigg(\abs{w_i}\,\abs{w_j}\Big(1-\cos\big(\frac{2\pi}{k}m(i-j)\big)\Big)\bigg)^{2^*/2}\right)^{2/2^*}\\
&& =
\sum_{i,j}\Big(1-\cos\big(\frac{2\pi}{k}m(i-j)\big)\Big)\frac{(\log\abs{x_i-x_j})^{\frac{N-2}{N}}}{\abs{x_i-x_j}^{N-2}}
\end{eqnarray*}
according to (\ref{secondo_asintotico}). Now, since
$\abs{x_i-x_j}\sim
R\big(1-\cos\big(\frac{2\pi}{k}(i-j)\big)\big)^{1/2}$, the sum
\begin{eqnarray*}
\lefteqn{\sum_{i,j}\Big(1-\cos\big(\frac{2\pi}{k}m(i-j)\big)\Big)\frac{\bigg(\log
\Big(R\big(1-\cos\big(\frac{2\pi}{k}(i-j)\big)\big)^{1/2}\Big)\bigg)^{\frac{N-2}{N}}}{R^{N-2}\Big(1-\cos\big(\frac{2\pi}{k}(i-j)\big)\Big)^{\frac{N-2}{2}}}}\\
&&\leq C(m)\, k \sum_{j}\frac{\bigg(\log
\Big(R\big(1-\cos\big(\frac{2\pi}{k}j\big)\big)^{1/2}\Big)\bigg)^{\frac{N-2}{N}}}{R^{N-2}\Big(1-\cos\big(\frac{2\pi}{k}j\big)\Big)^{N/2-2}}\\
&& \sim 2\,C(m)\,k^2 \int_0^{1/2}\frac{\bigg(\log
\Big(R\big(1-\cos\big(2\pi x\big)\big)^{1/2}\Big)\bigg)^{\frac{N-2}{N}}}{R^{N-2}\Big(1-\cos(2\pi x)\Big)^{N/2-2}}\,dx \\
&&\leq \frac{C(m)\,k^2}{R^{N-2}} \left\{
\begin{array}{ll}
O(\log R) \quad \textrm{if $N=4$}\\
O(\log R\log k) \quad \textrm{if $N=5$}\\
O\big((\log R\log k)^{\frac{N-2}{N}}k^{N-5}\big) \quad \textrm{if $N\geq 6$}\\
\end{array}\right.
\end{eqnarray*}
so that the second term (\ref{secondo termine}) is
\begin{eqnarray}\label{disuguaglianza secondo termine}
(\ref{secondo termine}) &\leq&\ C(m)\,k^{2/N}\frac{k^2}{R^{N-2}} \left\{
\begin{array}{ll}
O(\log R) \quad &\textrm{if $N=4$}  \\
O(\log R\log k) \quad &\textrm{if $N=5$} \\
O\big((\log R\log k)^{\frac{N-2}{N}}k^{N-5}\big) \quad &\textrm{if $N\geq 6$}\\
\end{array}\right. \nonumber \\
&=& \left\{
\begin{array}{ll}
O\big(\dfrac{k^{5/2}\log R}{R^2}\big) \quad &\textrm{if $N=4$}  \\
O\big(\dfrac{k^{12/5}\log R\log k}{R^{3}}\big) \quad &\textrm{if $N=5$}  \\
O\big(\dfrac{k^{N-3+2/N}(\log R\log k)^{\frac{N-2}{N}}}{R^{N-2}}\big) \quad &\textrm{if $N\geq 6$}\\
\end{array}\right.
\end{eqnarray}
which can be made $o\big(\dfrac{k^{N-1}}{R^{N-2}}\big)$ in every dimension for a suitable choice of the parameters
$R$ and $k$ (e.g. $k\sim R^\alpha$ with $0<\alpha<1$ since according to the hypothesis of lemma itself we need $k=o(R)$).

Provided the ratio $R/k$ is big enough, from equations (\ref{disuguaglianza susanna 2}),
(\ref{asintotico_alfa 2}) and (\ref{disuguaglianza secondo termine})
we get
$$
\int_{\R^N} \abs{\sum_{j=1}^k w_j}^{2^*}\ \geq \ k\,S^{N/2} +
2^*(1-\delta) \int_{\R^N}\sum_{i\neq j}\abs{w_i}^{2^*-1}\abs{w_j}\,,
$$
which in particular implies the thesis. \qed

\emph{\textbf{Proof of Lemma (\ref{scelta alfa,beta,gamma})}}
In order to have this quantity (positive) and less than 1, it is sufficient to have
\begin{enumerate}
 \item $\alpha/k$, $\gamma/k$ and $\beta/k$ small,
 \item $\alpha$ arbitrarly greater than $\beta$,
 \item $\alpha$ arbitrarly greater than $\gamma$.
\end{enumerate}
According to Lemma (\ref{lemma_asintotici}), Lemma (\ref{asintotico gamma totale}) and Remark (\ref{distanza centri}) we know
\begin{eqnarray*}\label{beta}
 \beta &=& \left\{\begin{array}{ll}
 O\big(k^2\displaystyle\frac{\log R}{R^2}\big) \quad \textrm{if $N=4$}\\
 O\big(\displaystyle\frac{k^2}{R^2}\big) \quad \textrm{for $N\geq5$;}
 \end{array}\right.\\
\gamma &=& \left\{\begin{array}{ll}
 O\big(\displaystyle\frac{k^2}{R^2}\big) \quad \textrm{if $N=4$}\\
 O\big(\displaystyle\frac{k^2}{R^3}\log k\big) \quad \textrm{for $N=5$}\\
 O\big(\displaystyle\frac{k^{N-3}}{R^{N-2}}\big) \quad \textrm{for $N\geq 6$}
 \end{array}\right. \label{gamma}\\
\alpha &=& O\big(k^2\dfrac{k^{N-2}}{R^{N-2}}\big)\label{alpha}.
\end{eqnarray*}
Let us fix the condition
\begin{equation}\label{alfa piccolo}
k^{(N-1)/(N-2)}=o(R)
\end{equation}
in order to have $\alpha/k$ small. Consequently we immediately find
the request 1 fulfilled. Moreover, we note this does not contradict
either the hypothesis of Lemma (\ref{lemma disuguaglianza den})
(rather, that is a consequence), or the conditions on equation
(\ref{disuguaglianza secondo termine}).

For what concerns request 2 and 3, we recall equation (\ref{stima alpha da sotto}) states the lower bound
$\alpha \gg k^{N-1}/R^{N-2}$.

Thus, we find request 3 satisfied as soon as $k\rightarrow\infty$.

About request 2, everything works without any additional hypothesis in dimension 4. In dimension $N\geq 5$,
we need $R=o(k^{(N-3)/(N-4)})$: we emphasize this does not contradict condition (\ref{alfa piccolo}) thanks to the order
$\frac{N-1}{N-2}<\frac{N-3}{N-4}$. \qed

\bigskip

As a natural question, letting $k\rightarrow\infty$, we wonder if
there exists any biradial solution: we mean a function belonging to
the space
\begin{eqnarray*}
\domdr=\{u\in \dom\ \textrm{s.t.}\ u(R(x_1,x_2),Sx_3)=u((x_1,x_2),x_3)\\
\indent\forall R\in SO(2)\,,\forall S\in SO(N-2)\}\ .
\end{eqnarray*}
In order to investigate this question, we set the problem
$$S_{A,a}^{r_1,r_2}=\inf_{u\in\domdr}\frac{\displaystyle\int_{\R^N}\abs{(i\nabla - \frac{A}{\abs{x}^2})u}^2-\int_{\R^N}\frac{a}{\abs{x}^2}\abs{u}^2}{\displaystyle\left(\int_{\R^N}\abs{u}^{2^*}\right)^{2/2^*}}\ ,$$
and we are able to prove
\begin{prop}\label{teoesistdrsol}
 There exists a biradial solution.
\end{prop}
\begin{proof}
 As we usually do, we consider a minimizing sequence $u_n$ to 
$S_{A,a}^{r_1,r_2}$ and Solimini's lemma in $\domdr$, since this 
is a closed subspace of \\$\dom$. As usual, we reconduce ourselves 
to $u_n-\Phi(\cdot+x_n)\rightarrow0$ in $\domdr$ with $\Phi\neq0$ 
and suppose by contradiction $\abs{x_n}\rightarrow+\infty$.

To preserve the symmetry, in Solimini's decomposition we will find
all the functions obtained by $\Phi$ with a rotation of a $2\pi/k$
angle, for $k\in\Z$ fixed. Thus, we can write
$u_n-\sum_{i=1}^k\Phi(\cdot+x_n^i)\rightarrow0$ in $\domdr$. Now,
following the same calculations in Theorem (\ref{s^kachieved}), we
obtain $S_{A,a}^{r_1,r_2}\geq S_{A,a}^k\geq k^{2/N}S$ that leads to
$S_{A,a}^{r_1,r_2}=+\infty$ choosing $k$ arbitrary large: a
contradiction. \qed
\end{proof}

\section{Aharonov-Bohm type potentials}

In dimension 2, an Aharonov-Bohm magnetic field is a $\delta$-type magnetic field. A vector potential associated to the Aharonov-Bohm magnetic field in $\R^2$ has the form
$$\A(x_1,x_2)=\left(\frac{-\alpha x_2}{\abs{x}^2}\, , \frac{\alpha x_1}{\abs{x}^2}\right) \qquad (x_1,x_2)\in\R^2$$
where $\alpha$ is the field flux through the origin. In this contest
we want to take account of Aharonov-Bohm type potentials in $\R^N$,
for $N\geq4$:
$$
 \A(x_1,x_2,x_3)=\left(\frac{-\alpha x_2}{x_1^2+x_2^2}\, , \frac{\alpha x_1}{x_1^2+x_2^2}\, , 0\right) \qquad (x_1,x_2)\in\R^2\, , x_3\in\R^{N-2}\ ,
$$
paying special attention now the singular set is a whole subspace of
$\R^N$ with codimension 2.

\subsection{Hardy-type inequality}

In order to study minimum problems and therefore the quadratic form
associated to this kind of potentials, we need a Hardy-type
inequality. We know by \cite{LW99} that a certain Hardy-type
inequality holds for Aharonov-Bohm vector potentials in $\R^2$, that
is
$$
 \int_{\R^2}\frac{\abs{\varphi}^2}{\abs{x}^2}\ \leq\ C\int_{\R^2}\abs{(i\nabla - \A)\varphi}^2 \qquad \forall\varphi\in\cinf{\R^2\setminus\{0\}}\ ,
$$
where the best constant $C$ is
\begin{equation}\label{Chardy}
H=\left(\min_{k\in\Z}\abs{k-\Phi_\A}\right)^2\,.
\end{equation}
Here $\Phi_\A$ denotes the field flux around the origin
$$\Phi_\A=\frac{1}{2\pi}\int_{0}^{2\pi}\A(\cos t,\sin t)\cdot(-\sin t,\cos t)\,dt\ .$$
One can generalize this result and gain a similar inequality to the
Aharonov-Bohm potentials in $\R^N$, simply separating the integrals:
for all $\varphi\in\cinf{\R^N\setminus\{x_1=x_2=0\}}$ one has
\begin{eqnarray}\label{hardy}
 \int_{\R^N}\frac{\abs{\varphi}^2}{x_1^2+x_2^2}=\int_{\R^{N-2}}\int_{\R^{2}}\frac{\abs{\varphi}^2}{x_1^2+x_2^2}\,dx_1\,dx_2\,\,dx_3
\leq H \int_{\R^N}\abs{(i\nabla - \A)\varphi}^2,
\end{eqnarray}
where $H$ is defined in (\ref{Chardy}). Now a natural question
arises: is $H$ the best constant for inequality (\ref{hardy})? In
other words, is $H$ the infimum of the Rayleigh quotient?

\begin{prop}
 The best constant for the inequality (\ref{hardy}) is exactly (\ref{Chardy}).
\end{prop}
\begin{proof}
 To prove this, we consider the approximating sequence $u_n$ to (\ref{Chardy}) in $\R^2$. We can choose this sequence bounded in $L^2(\R^2)$ norm, thanks to the homogeneity of the quotient under dilation.

We claim there exists a sequence of real-valued functions
$(\eta_n)_n\subset\\ \cinf{\R^{N-2}}$ such that
$\int_{\R^{N-2}}\abs{\nabla\eta_n}^2\,\longrightarrow 0$ and
$\int_{\R^{N-2}}\eta_n^2\,\longrightarrow\,+\infty$ as $n\rightarrow
+\infty$. We can namely consider a real radial function such that
$\eta_n\equiv1$ in $B_R(0)$ and $\eta_n\equiv0$ in
$\R^{N-2}\setminus B_{R+n^\alpha}(0)$, with
$\abs{\nabla\eta_n}\sim\frac{1}{n^\alpha}$, for a suitable
$\alpha>0$ (e.g. $\alpha>\frac{N-2}{2}$).

Now we consider the sequence $v_n(x_1,x_2,x_3)=u_n(x_1,x_2)\eta_n(x_3)$ where $x_3$ as usual denotes the whole set of variables in $\R^{N-2}$, and test the quotient over this sequence:
\begin{eqnarray*}
 \frac{\displaystyle\int_{\R^N}\abs{(i\nabla-\A)v_n}}{\displaystyle\int_{\R^N}\frac{\abs{v_n}^2}{x_1^2+x_2^2}}=
\frac{\displaystyle\int_{\R^N}\abs{\nabla v_n}^2-2Re\int_{\R^N}\A v_n\cdot \nabla\overline{v}_n
+\int_{\R^N}\abs{\A}^2\abs{v_n}^2}{\displaystyle\int_{\R^{N-2}}\abs{\eta_n}^2\int_{\R^2}\frac{\abs{u_n}^2}{x_1^2+x_2^2}},
\end{eqnarray*}
where the numerator is
\begin{eqnarray}
 \int_{\R^N}\eta^2_n\abs{\nabla u_n}^2
+\int_{\R^N}\abs{\A}^2\abs{\eta_n}^2\abs{u_n}^2
-2Re\int_{\R^N} \eta_n^2u_n\A\cdot\nabla\overline{u}_n \nonumber\\
+2\,Re\int_{\R^N}u_n\eta_n\nabla \eta_n\cdot\nabla \overline{u}_n
+\int_{\R^N}u_n^2\abs{\nabla\eta_n}^2
-2Re\int_{\R^N} \abs{u_n}^2\eta_n\A\cdot\nabla\eta_n\,. \label{infinitesimi numeratore}
\end{eqnarray}

About the second line (\ref{infinitesimi numeratore}) in the numerator, via H\"{o}lder inequality we have
\begin{eqnarray*}
 \lefteqn{\abs{\int_{\R^N}u_n\eta_n\nabla \eta_n\cdot\nabla \overline{u}_n}
\leq \left(\int_{\R^N}\abs{u_n}^2\abs{\nabla\eta_n}^2\right)^{1/2} \left(\int_{\R^N}\eta_n^2\abs{\nabla u_n}^2\right)^{1/2}}\\
&&= \left(\int_{\R^{N-2}}\abs{\nabla\eta_n}^2\right)^{1/2}
\left(\int_{\R^2}\abs{u_n}^2\right)^{1/2}
\left(\int_{\R^{N-2}}\eta_n^2\right)^{1/2}
\left(\int_{\R^2}\abs{\nabla u_n}^2\right)^{1/2}
\end{eqnarray*}
and
\begin{eqnarray*}
\lefteqn{\abs{\int_{\R^N}\A
u_n\eta_n\cdot\overline{u}_n\nabla\eta_n}
\leq\left(\int_{\R^N}\abs{\nabla\eta_n}^2\abs{u_n}^2\right)^{1/2}\left(\int_{\R^N}\abs{\A}^2\abs{\eta_n}^2\abs{u_n}^2\right)^{1/2}}\\
&&=
\left(\int_{\R^{N-2}}\abs{\nabla\eta_n}^2\int_{\R^2}\abs{u_n}^2\right)^{1/2}
\left(\int_{\R^{N-2}}\abs{\eta_n}^2\int_{\R^2}\abs{\A}^2\abs{u_n}^2\right)^{1/2}
\end{eqnarray*}
Therefore the Rayleigh quotient is reduced to
\begin{eqnarray*}
\lefteqn{\frac{{\displaystyle\int_{\R^2}\abs{\nabla
u_n}^2+\int_{\R^2}\abs{\A}^2\abs{u_n}^2}-2Re\displaystyle\int_{\R^2}\A
u_n\cdot\nabla\overline{u}_n}{\displaystyle\int_{\R^2}\frac{\abs{u_n}^2}{x_1^2+x_2^2}}\ +}\\
&&+\ \frac{\displaystyle2\,Re\int_{\R^N}u_n\eta_n\nabla
\eta_n\cdot\nabla
\overline{u}_n}{\displaystyle\int_{\R^{N-2}}\abs{\eta_n}^2\int_{\R^2}\frac{\abs{u_n}^2}{x_1^2+x_2^2}}
+\
\frac{\displaystyle\int_{\R^2}\abs{u_n}^2\int_{\R^{N-2}}\abs{\nabla\eta_n}^2}{\displaystyle\int_{\R^{N-2}}\abs{\eta_n}^2\int_{\R^2}\frac{\abs{u_n}^2}{x_1^2+x_2^2}}-\ \frac{\displaystyle2Re\int_{\R^N} \abs{u_n}^2\eta_n\A\cdot\nabla\eta_n}{\displaystyle\int_{\R^{N-2}}\abs{\eta_n}^2\int_{\R^2}\frac{\abs{u_n}^2}{x_1^2+x_2^2}}\\
\lefteqn{=\ H + o(1)}
\end{eqnarray*}
thanks to the properties of the sequence $\eta_n$. \qed
\end{proof}

\subsection{Variational setting}

We have seen before the quadratic form associated to
$\frac{A}{\abs{x}}$-type potentials is equivalent to the Dirichlet
form. On the contrary, we will see in case of Aharonov-Bohm
potentials it is stronger than the Dirichlet form, and consequently
the function space is a proper subset of $\dom$.

Indeed, for any $\varphi\in\cinf{\R^N\setminus\{x_1=x_2=0\}}$ we have the simple inequivalence
\begin{eqnarray*}
 \int_{\R^N}\abs{\nabla\varphi}^2&=&\int_{\R^N}\abs{(i\nabla-\A+\A)\varphi}^2 \leq C\left(\int_{\R^N}\abs{(i\nabla-\A)\varphi}^2+\int_{\R^N}\abs{\A}^2\abs{\varphi}^2\right) \\
&\leq& C\int_{\R^N}\abs{(i\nabla-\A)\varphi}^2
\end{eqnarray*}
thanks to Hardy-type inequality proved above.

It is immediate to see by this remark
$$\mathcal{H}_\A\circeq\overline{\cinf{\R^N\setminus\{x_1=x_2=0\}}}^{\int_{\R^N}\abs{(i\nabla-\A)\varphi}^2}\subseteq\dom\ .$$

To prove the strict inclusion it is sufficient to show a function
lying in $\dom$ but not in $\mathcal H_\A$. One can choose for
example $\varphi(x_1,x_2,x_3)=p(x_1,x_2,x_3)\abs{x}^{(-N+1)/2}$,
where $p$ is a cut-off function which is identically 0 in
$B_\varepsilon(0)$ and identically 1 in $\R^N\setminus
B_{2\varepsilon}(0)$: we have
$\abs{\nabla\varphi}^2\sim\abs{x}^{-N-1}$ which is integrable in
$\R^N\setminus B_\varepsilon(0)$, whereas
$\frac{\varphi}{x_1^2+x_2^2}$ is not, since $\varphi$ is far from 0
near the singular set.

\begin{rem}
Of course $\mathcal H_\A$ is a closed subspace of $\dom$. This is a
straightforward consequence of the density of
$\cinf{\R^N\setminus\{x_1=x_2=0\}}$ in $\mathcal H_A$ and the
relation between the two quadratic forms. Then Solimini's Theorem
(\ref{solimini}) holds also in this space.
\end{rem}

\smallskip

Following what we did in the previous case, we state the following
\begin{lem}\label{lemmadivergenzatraslazioneAB}
 Let $x_n$ a sequence of points such that $\abs{({x_n}^1,{x_n}^2)}\rightarrow\infty$ as $n\rightarrow\infty$. Then for any $u\in\mathcal H_\A$ as $n\rightarrow\infty$ we have
$$
\frac{\displaystyle\int_{\R^{N}}\bigg\vert\left(i\nabla-\A\right)u(\cdot+x_n)\bigg\vert^2
-\int_{\R^{N}}\frac{a(\theta)}{x_1^2+x_2^2}\,\abs{u(\cdot+x_n)}^2}{\bigg(\displaystyle\int_{\R^N}\abs{u}^{2^*}\bigg)^{2/2^*}}
\rightarrow\frac{\displaystyle\int_{\R^N}\abs{\nabla
u}^2}{\left(\displaystyle\int_{\R^N}\abs{u}^{2^*}\right)^{2/{2^*}}}.
 $$
\end{lem}
\begin{proof}
We can follow the proof of Lemma (\ref{lemmadivergenzatraslazione}) noting here the singularity involves only the first two variables. \qed
\end{proof}

So that we immediately have the following property for $S_{\A,a}$:
\begin{prop}
If the electric potential $a$ is invariant under translations in
$\R^{N-2}$ (as the magnetic vector potential actually is), the
related minimum problem leads to
$$S_{\A,a}=\inf_{u\in\mathcal H_A\setminus\{0\}}\frac{\displaystyle\int_{\R^N}\abs{(i\nabla-\A)u}^2-\int_{\R^N}\frac{a}{x_1^2+x_2^2}\abs{u}^2}{\displaystyle\left(\int_{\R^N}\abs{u}^{2^*}\right)^{2/2^*}}\ \ \leq S\ .$$
\end{prop}
\begin{proof}
We follow the proof of Proposition (\ref{SA<S}) taking into account Lemma (\ref{lemmadivergenzatraslazioneAB}). \qed
\end{proof}

\subsection{Achieving the Sobolev constant}

As in the previous case, we state the following
\begin{prop}
 If $S_{\A,a}<S$ then $S_{\A,a}$ is achieved.
\end{prop}
\begin{proof}
Let us consider a minimizing sequence $u_n$. As we have already
noticed, we can apply Solimini's Theorem (\ref{solimini}) in order
to get $u_n-\Phi(\cdot+x_n)\,\longrightarrow\,0$ in $\mathcal H_\A$
as $n\rightarrow+\infty$ following the same argument as above. Once
again, we argue by contradiction. Supposing the sequence $x_n$
diverges means it is divergent with respect to the two first
variables $(x_n^1,x_n^2)$, since both the electric and magnetic
potential are invariant under translations in $\R^{N-2}$. This means
that if $x_n$ diverges to infinity with respect to $x_n^3$, this is
a convergence case.

So, by contradiction let us analyze only the case
$\abs{(x_n^1,x_n^2)}\rightarrow+\infty$. According to Lemma
(\ref{lemmadivergenzatraslazioneAB}) we have
\begin{eqnarray*}
\lefteqn{\frac{\displaystyle\int_{\R^N}\abs{(i\nabla-\A)u_n}^2-\int_{\R^N}\frac{a}{x_1^2+x_2^2}\abs{u_n}^2}{\displaystyle\left(\int_{\R^N}\abs{u_n}^{2^*}\right)^{2/2^*}}}\\
&&=\frac{\displaystyle\int_{\R^N}\abs{(i\nabla-\A)\Phi(\cdot+x_n)}^2-\int_{\R^N}\frac{a}{x_1^2+x_2^2}\abs{\Phi(\cdot+x_n)}^2}{\displaystyle\left(\int_{\R^N}\abs{\Phi(\cdot+x_n)}^{2^*}\right)^{2/2^*}}+o(1) \\
&&\geq\frac{\displaystyle\int_{\R^N}\abs{\nabla\abs{\Phi(\cdot+x_n)}}^2-\int_{\R^N}\frac{a}{x_1^2+x_2^2}\abs{\Phi(\cdot+x_n)}^2}{\displaystyle\left(\int_{\R^N}\abs{\Phi(\cdot+x_n)}^{2^*}\right)^{2/2^*}}+o(1)
\geq S+o(1)\,.
\end{eqnarray*}
Thus we obtain $S_{\A,a}\geq S$, a contradiction. \qed
\end{proof}

\subsection{Symmetry of solutions}

We introduce the space
$$\mathcal H_\A^k=\{u(z,y)\in\mathcal H_\A\ \textrm{s.t.}\ u(\e^{i\frac{2\pi}{k}}z,y)=u(z,\abs{y})\}\ ,$$
which is a closed subspace of $\mathcal H_\A$, so Solimini's Theorem (\ref{solimini}) holds in it.

We should suppose that the magnetic potential $\A$ is invariant under the $\Z_k\times SO(N-2)$-group action on $\mathcal H_\A$, as in (\ref{A-Zksymm}). But in this case, the magnetic vector potential enjoys this symmetry thanks to its special form. On the other hand, we choose the electric potential $a$ as a negative constant.

Following the same proof as in the previous case, we can state the
following
\begin{prop}
 If $S_{\A,a}^k<k^{2/N}S$ then $S_{\A,a}^k$ is achieved.
\end{prop}

Now we look for sufficient conditions to have $S_{\A,a}^k<k^{2/N}S$.

The idea is again to check the quotient over a suitable sequence of
test functions. We choose as well $\sum_{i=1}^kw_j$, where $w_j$ are
defined in (\ref{wj}) and the lines above it. Of course, we need to
multiply them by a cut-off function
$\varphi(x_1,x_2,x_3)=\varphi(x_1,x_2)=\varphi(\sqrt{x_1^2+x_2^2})=\varphi(\rho)$,
in order to obtain the necessary integrability near the singular set.

\begin{lem}\label{valangadicontiAB}
 Choosing $R$ big enough in (\ref{wj}), the quotient evalueted over $\varphi\sum_{i=1}^kw_j$ is strictly less than $k^{2/N}S$, and so the infimum $S^k_{\A,a}$.
\end{lem}
\begin{proof}
Let us check the quotient over $\varphi\sum_{j=1}^k w_j$. In
$$
\int_{\R^N}\Big\vert\nabla(\varphi\sum_{j=1}^k w_j)\Big\vert^2
+\frac{\alpha^2-a}{x_1^2+x_2^2}\varphi^2\Big\vert\sum_{j=1}^k w_j\Big\vert^2
-2Re\Bigg\{i \nabla(\varphi\sum_{j=1}^k w_j)\cdot \A\varphi\sum_{j=1}^k \overline{w}_j\Bigg\}
$$
we study term by term.
First of all
$$
 \Big|\nabla(\varphi\sum_{j=1}^k w_j)\Big|^2
=\abs{\nabla\varphi}^2\Big\vert\sum_{j=1}^k w_j\Big\vert^2+\varphi^2\Big|\nabla(\sum_{j=1}^k w_j)\Big|^2+2\,Re\left\{\varphi\nabla\varphi\cdot\sum_{j,l}\nabla w_j \overline{w}_l\right\}
$$
and
\begin{eqnarray*}
 \int_{\R^N}\varphi^2\Big|\nabla(\sum_{j=1}^k w_j)\Big|^2
&=&\int_{\R^N}\Big|\nabla(\sum_{j=1}^k w_j)\Big|^2-\int_{\R^N}(1-\varphi^2)\Big|\nabla(\sum_{j=1}^k w_j)\Big|^2 \\ \nonumber
&=&kS^{N/2}+\int_{\R^N}Re\Big\{\sum_{j\neq l}\abs{w_j}^{2^*-2}w_j\overline{w}_l\Big\}-\int_{\R^N}(1-\varphi^2)\Big|\nabla(\sum_{j=1}^kw_j)\Big|^2\ .
\end{eqnarray*}
Secondly
\begin{eqnarray*}
 \nabla(\varphi\sum_{j=1}^k w_j)\cdot \A\varphi\sum_{j=1}^k \overline{w}_j
=\varphi\nabla\varphi\cdot \A\Big|\sum_{j=1}^k w_j\Big|^2+\varphi^2\sum_{j,l}\nabla w_j\cdot \A\, \overline{w}_l\ .
\end{eqnarray*}
So, the quadratic form is the following
\begin{eqnarray*}
 \lefteqn{kS^{N/2}+\int_{\R^N}Re\Big\{\sum_{j\neq l}\abs{w_j}^{2^*-2}w_j\overline{w}_l\}
-\int_{\R^N}(1-\varphi^2)\Big|\nabla(\sum_{j=1}^k w_j)\Big|^2
+\int_{\R^N}\abs{\nabla\varphi}^2\Big|\sum_{j=1}^k w_j\Big|^2}\\
&&+2\int_{\R^N}Re\Big\{\varphi\nabla\varphi\cdot\sum_{j,l}\nabla w_j\,\overline{w}_l\Big\}
+\int_{\R^N}\frac{\alpha^2-a}{x_1^2+x_2^2}\varphi^2\Big|\sum_{i=1}^k w_j\Big|^2 \\
&&-2\int_{\R^N}Re\Big\{i\,\varphi\nabla\varphi\cdot \A\Big|\sum_{j=1}^k w_j\Big|^2\Big\}
-2\int_{\R^N}Re\Big\{i\,\varphi^2\sum_{j,l}\nabla w_j\cdot \A\, \overline{w}_l\Big\} \\
\lefteqn{\leq kS^{N/2}+\int_{\R^N}Re\Big\{\sum_{j\neq l}\abs{w_j}^{2^*-2}w_j\overline{w}_l\}
+\int_{\R^N}\abs{\nabla\varphi}^2\Big|\sum_{j=1}^k w_j\Big|^2}\\
&&+2\left(\int_{\R^N}\varphi^2\Big|\sum_{j=1}^k\nabla w_j\Big|^2\right)^{1/2}\left(\int_{\R^N}\abs{\nabla\varphi}^2\Big|\sum_{j=1}^k w_j\Big|^2\right)^{1/2}\\
&&+\int_{\R^N} \frac{\alpha^2-a}{x_1^2+x_2^2} \varphi^2\Big|\sum_{j=1}^k w_j\Big|^2 \\ 
&&+2\left(\int_{\R^N}\abs{\nabla\varphi}^2\Big|\sum_{j=1}^k w_j\Big|^2\right)^{1/2}\left(\int_{\R^N}\varphi^2\abs{\A}^2\Big|\sum_{j=1}^k w_j\Big|^2\right)^{1/2}\\
&&-2\int_{\R^N}Re\Big\{i\,\varphi^2\sum_{j,l}\nabla w_j\cdot \A\, \overline{w}_l\Big\}\ ,
\end{eqnarray*}
whereas the denominator
\begin{eqnarray}\label{disuguaglianza denominatore AB}
 \int_{\R^N}\varphi^{2^*}\Big|\sum_{j=1}^k w_j\Big|^{2^*}&=&\int_{\R^N}\Big|\sum_{j=1}^k w_j\Big|^{2^*}
-\int_{\R^N}(1-\varphi^{2^*})\Big|\sum_{j=1}^k w_j\Big|^{2^*} \nonumber \\
&\geq& kS^{N/2}+2^*(1-\delta/2)\int_{\R^N}Re\sum_{j\neq l}\abs{w_j}^{2^*-2}w_j\,\overline{w}_l \nonumber \\
&& -\int_{\R^N}(1-\varphi^{2^*})\Big|\sum_{j=1}^k w_j\Big|^{2^*}\ .
\end{eqnarray}
To simplify the notation, we set $R=\sqrt{({x_j}^1)^2+({x_j}^2)^2}$
and we have
\begin{eqnarray}
 \alpha&=&\int_{\R^N}Re\Big\{\sum_{j\neq l}\abs{w_j}^{2^*-2}w_j\,\overline{w}_l\Big\}
\left\{ \begin{array}{ll}
         \leq O\big(\dfrac{k^{N}}{R^{N-2}}\big)\\
     \gg \dfrac{k^{N-1}}{R^{N-2}}
        \end{array} \right.
 \label{alfa2}\\
\beta&=&\int_{\R^N}\abs{\nabla\varphi}^2\Big|\sum_{j=1}^k w_j\Big|^2\ \leq\ O\Big(\frac{k^2}{R^{2N-4}}\Big) \nonumber\\
\gamma&=&2\left(\int_{\R^N}\varphi^2\Big|\nabla(\sum_{j=1}^k w_j)\Big|^2\int_{\R^N}\abs{\nabla\varphi}^2\Big|\sum_{j=1}^k w_j\Big|^2\right)^{1/2}\ \leq \ O\Big(\frac{k^{3/2}}{R^{N-2}}\Big) \nonumber\\
\eta&=&\int_{\R^N}\frac{\alpha^2-a}{x_1^2+x_2^2}\varphi^2\Big|\sum_{j=1}^k w_j\Big|^2\
\leq\ \left\{ \begin{array}{ll}
    O\big(\dfrac{k^2}{R^2}\log R\big) \quad \textrm{if $N=4$} \\
    O\big(\dfrac{k^2}{R^2}\big) \quad \textrm{if $N\geq 5$} \nonumber
   \end{array} \right. \\
\xi&=&2\left(\int_{\R^N}\abs{\nabla\varphi}^2\Big|\sum_{j=1}^k w_j\Big|^2\right)^{1/2}\left(\int_{\R^N}\varphi^2\abs{\A}^2\Big|\sum_{j=1}^k w_j\Big|^2\right)^{1/2}\nonumber\\
&\leq&\ \left\{ \begin{array}{ll}
            O\big(\dfrac{k^2}{R^{3}}\log^{1/2}R\big) \quad \textrm{if $N=4$}  \\
        O\big(\dfrac{k^2}{R^{N-1}}\big) \quad \textrm{if $N\geq 5$} \label{xi2}
           \end{array} \right. \\
\psi&=&\int_{\R^N}(1-\varphi^{2^*})\Big|\sum_{j=1}^k w_j\Big|^{2^*}\ \leq\ O\big(\frac{k^2}{R^{2N}}\big)\nonumber
\end{eqnarray}
while for the last term we have
\begin{equation}\label{zeta2}
 \abs{\int_{\R^N} Re \Big\{ i\,\varphi^2 \sum_{j,l}\nabla w_j \cdot \A \overline{w}_l \Big\}} \leq O\big(\dfrac{k^{N-3}}{R^{N-2}}\big)
\end{equation}
since Lemma (\ref{asintotico gamma totale}) fits also in this case with the suitable modifications.
In (\ref{alfa2}) the symbol $\gg$ stands for $\alpha$ has order strictly greater than ${k^{N-1}}/{R^{N-2}}$.

We note all these quantities $\alpha$, $\beta$, $\gamma$, $\eta$, $\xi$, $\zeta$ can be chosen small simply taking
the quotient $k^{N-1}/R^{N-2}$ small (namely $k^{N-1}/R^{N-2}=\varepsilon$), as we can deduce from (\ref{alfa2}),
\dots, (\ref{zeta2}).

Moreover, we see $\psi=o(\alpha)$, so that we can improve estimate (\ref{disuguaglianza denominatore AB}) and state
$$
 \int_{\R^N}\varphi^{2^*}\Big|\sum_{j=1}^k w_j\Big|^{2^*}
\geq kS^{N/2}+2^*(1-\delta/2)\int_{\R^N}Re\sum_{j\neq l}\abs{w_j}^{2^*-2}w_j\,\overline{w}_l
$$
for a different $\delta$ from above.

With the simplified notation, the quotient takes the form
\begin{eqnarray*}
 \frac{kS^{N/2}+\alpha+\beta+\gamma+\eta+\xi+\zeta}{\big(kS^{N/2}+2^*(1-\delta/2)\alpha\big)^{2/2^*}}=k^{2/N}S\,\frac{\displaystyle1+\frac{1}{kS^{N/2}}(\alpha+\beta+\gamma+\eta+\xi+\zeta)}{\displaystyle\Big(1+\frac{2^*(1-\delta/2)}{kS^{N/2}}\alpha\Big)^{2/2^*}}\ .
\end{eqnarray*}
Expanding the quotient in first order power series, it is asymptotic
to
\begin{eqnarray*}
 k^{2/N}S && \Big(1+\frac{1}{kS^{N/2}}(\alpha+\beta+\gamma+\eta+\xi+\zeta)\Big)\Big((1-\frac{2(1-\delta/2)}{kS^{N/2}}\alpha\Big)\\
\sim k^{2/N}S && \Big\{1+\frac{1}{kS^{N/2}}\big(\beta+\gamma+\eta+\xi+\zeta\big)\\
&& +\frac{1}{kS^{N/2}}\alpha\Big(-1+\delta+\frac{1}{kS^{N/2}}\big(\beta+\gamma+\eta+\xi+\zeta\big)\Big)\Big\}
\end{eqnarray*}
Now, in order to have the coefficient of $k^{2/N}S$ strictly less than 1, it is sufficient that $\beta$, $\gamma$, $\eta$, $\xi$, $\zeta$ are $o(k^{N-1}/R^{N-2})$.
Taking into account (\ref{alfa2}), $\dots$, (\ref{xi2}) and (\ref{zeta2})
we see it is sufficient choosing
$k$ as in the previous case of $\dfrac{A}{\abs{x}}$-type potentials. \qed
\end{proof}

\bigskip

As we made in the previous section, we wonder if there exists any
biradial solution, meaning a function belonging to the space
\begin{eqnarray*}
\mathcal H_\A^{r_1,r_2}=\{u\in\mathcal H_\A\ \textrm{s.t.}\
&u(R(x_1,x_2),Sx_3)=u((x_1,x_2),x_3)\\
&\forall R\in SO(2)\,,\forall S\in SO(N-2)\}\ .
\end{eqnarray*} In order to
investigate this question, we set the problem
$$S_{\A,a}^{r_1,r_2}=\inf_{u\in\mathcal H_\A^{r_1,r_2}}\frac{\displaystyle\int_{\R^N}\abs{(i\nabla - A)u}^2-\int_{\R^N}\frac{a}{x_1^2+x_2^2}\abs{u}^2}{\displaystyle\left(\int_{\R^N}\abs{u}^{2^*}\right)^{2/2^*}}\,,$$
and we state
\begin{prop}\label{teoesistdrsolAB}
 There exists a biradial solution.
\end{prop}
\begin{proof}
We follow the proof of Proposition (\ref{teoesistdrsol}) that fits also in this case with the suitable modifications. \qed

\end{proof}

\section{Symmetry breaking}

In order to proceed in our analysis, we need to recall a result proved in \cite{ATdr09}:
\begin{thm}(\cite{ATdr09})\label{doublyradialth}
 Suppose $u=u(r_1,r_2)$ (where $r_1=\sqrt{x_1^2+x_2^2}$ and $r_2=\sqrt{x_3^2+\cdots+x_N^2}$) is a solution to
$$
 -\Delta u-\frac{a}{\abs{x}^2}u=f(x,u)
$$
with $a\in\R^-$ and $f:\R^N\times\C\rightarrow\C$ being a Carath\'{e}odory function, $C^1$ with respect to $z$, such that it satisfies the growth restriction
$$\abs{f'_z(x,z)} \leq C(1+\abs{z}^{2^*-2})$$
for a.e. $x\in\R^N$ and for all $z\in\C$.

If the solution $u$ has Morse index $m(u)\leq1$, then $u$ is a radial solution, that is $u=u(r)$ where $r=\sqrt{x_1^2+\cdots+x_N^2}$.
\end{thm}

Roughly speaking, in particular this theorem states any biradial
solution to (\ref{eqintro}) found as minimizer of the Sobolev
quotient is in fact completely radial, since such a solution and
equation (\ref{eqintro}) satisfies the hypothesis of Theorem
(\ref{doublyradialth}). This leads to a so-called solutions'
symmetry breaking. We point out the notation we used up to now:
\begin{defn}
 $S_{A,a}^{r_1,r_2}$ is the minimum of the Rayleigh quotient related to the magnetic Laplacian over all the biradial functions in $\dom$;

$S_{0,a}^{r_1,r_2}$ is the minimum of the Rayleigh quotient related
to the usual Laplacian over all the biradial functions in $\dom$;

$S_{0,a}^{\textrm rad}$ is the minimum of the Rayleigh quotient
related to the usual Laplacian over all the radial functions in
$\dom$;

$S_{0,a}^{k}$ is the minimum of the Rayleigh quotient related to the
usual Laplacian over all the functions in $\domk$;

$S_{A,a}^{k}$ is the minimum of the Rayleigh quotient related to the
magnetic Laplacian over all the functions in $\domk$;

$S$ is the usual Sobolev constant for the embedding
$\dom\hookrightarrow L^{2^*}(\R^N)$.
\end{defn}

So we can collect our information on these quantities and write the
following chain of relations:
$$
 S_{A,a}^{r_1,r_2} \geq S_{0,a}^{r_1,r_2} = S_{0,a}^{\textrm rad} \geq k^{2/N}S > S_{A,a}^{k}
$$
where the first inequality holds thanks to diamagnetic inequality; the fact $S_{0,a}^{r_1,r_2} = S_{0,a}^{\textrm rad}$ is a straightforward consequence of Theorem (\ref{doublyradialth}); the second inequality is proved in \cite{Ter96}, Section 6 for sufficiently large values of $\abs{a}$; and the last one is proved in Lemma (\ref{valangadiconti}).

\begin{rem}\emph{Symmetry breaking for Aharonov-Bohm electromagnetic potentials.}
We note the same facts hold also for Aharonov-Bohm electromagnetic
fields. Indeed, the diamagnetic inequivalence holds also for them
with the \emph{same} best constant, because the Hardy constant is
the same (see Section 4.1); moreover, $\frac{a}{x_1^2+x_2^2}\geq
\frac{a}{\abs{x}^2}$ for $a>0$. So we can rewrite
$$
S_{\A,a}^{r_1,r_2} \geq
S_{0,a}^{r_1,r_2} = S_{0,a}^{\textrm rad} \geq k^{2/N}S >
S_{\A,a}^{k}
$$
where the last inequivalence has been proved in Lemma (\ref{valangadicontiAB}).
\end{rem}

\end{document}